\documentclass{mcom-l}
\usepackage{amsfonts}
\usepackage{amsmath}
\usepackage{amsthm,amssymb}
\usepackage{fullpage,relsize,cite,graphicx,mathrsfs,comment,caption,subcaption,mathtools}
\usepackage{color}

\usepackage{verbatim}

\newtheorem{lemma}{Lemma}
\newtheorem{corollary}{Corollary}
\newtheorem{definition}{Definition}
\newtheorem{theorem}{Theorem}
\newtheorem{problem}{Problem}

\newtheorem{proposition}{Proposition}
\newtheorem{remark}{Remark}
\newtheorem{assumption}{Assumption}

\newcommand{\DeltaG}{\Delta_\Gamma}
\newcommand{\nablaG}{\nabla_\Gamma}

\newcommand{\unit}{\texttt{1}\!\!\texttt{l}}

\def \D #1{\underline{D}_{#1}}

\usepackage{xifthen}
\usepackage{marginnote}
\newcounter{todocounter}
\DeclareRobustCommand{\MyChange}[3][\empty]{%
  {\color{#2}#3}
  \ifthenelse{\isempty{#1}}{}
  {%
    \addtocounter{todocounter}{1}%
    \ifmmode%
        {\color{#2}\text{$^{\framebox{\arabic{todocounter}}}$}}%
    \else%
        {\color{#2}\text{$^{\arabic{todocounter}}$}}%
    \fi%
    \marginpar{\textcolor{#2}{$^{\arabic{todocounter}}$\textnormal{#1}}}%
  }%
}
\definecolor{lightgray}{rgb}{0.7,0.7,0.7}
\usepackage[normalem]{ulem}

\definecolor{ghcol}{rgb}{0.7,0.7,0}

\numberwithin{equation}{section}
\numberwithin{lemma}{section}
\numberwithin{corollary}{section}
\numberwithin{theorem}{section}
\numberwithin{proposition}{section}
\numberwithin{problem}{section}
\numberwithin{remark}{section}
\numberwithin{definition}{section}
\numberwithin{assumption}{section}

\title{Second order splitting for a class of fourth order equations}

\author{Charles M. Elliott}
\address{Mathematics Institute, Zeeman Building, University of Warwick, Coventry. CV4 7AL. UK}
\email{C.M.Elliott@warwick.ac.uk}
\author{Hans Fritz}
\address{Fakult{\"a}t f{\"u}r Mathematik,
Universit{\"a}t Regensburg,
93040 Regensburg, Germany}
\email{hans.fritz@ur.de}
\author{Graham Hobbs}
\address{Mathematics Institute, Zeeman Building, University of Warwick, Coventry. CV4 7AL. UK}\email{graham.hobbs@warwickgrad.net}
\thanks{
The work of CME was partially supported by the Royal Society via a Wolfson Research Merit Award. 
HF  thanks the Alexander von Humboldt Foundation, Germany, for their financial
support by a Feodor Lynen Research Fellowship in collaboration with the University of Warwick,
UK. The research of GH was funded by the Engineering and Physical Sciences Research Council grant  EP/H023364/1
 under the MASDOC centre for doctoral training at the University of Warwick.}
\date{}

\begin{document}
\setlength\parindent{1cm}

\begin{abstract}
{We formulate a well-posedness and approximation theory for a class of generalised saddle point problems. In this way we develop an approach to a class of fourth order elliptic partial differential equations using  the idea of splitting into coupled second order equations.
Our main motivation is to treat certain fourth order equations on closed surfaces arising in the modelling of 
 biomembranes but the approach may be applied more generally. In particular  we are interested in equations with non-smooth right hand sides and  operators which have non-trivial kernels.The theory for well posedness and approximation is presented 
 in an abstract setting. Several examples are described together with some numerical experiments.}
\end{abstract}

\maketitle

\section{Introduction}
We study  the well posedness and approximation of a generalised linear  saddle point problem in reflexive Banach spaces using three bilinear forms:  find $(u,w) \in X \times Y$ such that
\begin{equation}
\begin{split}
c(u,\eta) + b(\eta,w) &= \langle f, \eta \rangle \quad \forall \eta \in X, \\
b(u,\xi) - m(w,\xi) &= \langle g, \xi \rangle \quad \forall \xi \in Y.
\end{split}
\end{equation}

Our assumptions on the bilinear forms and spaces will be detailed in Section  \ref{AbsSplit}. First we give some context to this general problem within the existing literature. If we were to set $m=0$ the resulting saddle point problem is well studied, see for example \cite{ErnGue13}, and the assumptions we will make  on $b$ and $c$ are sufficient to show well posedness. The $m \neq 0$ case is examined in \cite{CiaHuaZou03, KelLiu96, BofBreFor13}. In these papers  well posedness is shown under a different set of assumptions to ours. In contrast to our assumptions  only one of the inf sup conditions is required for $b$ and $m$ has a weaker coercivity assumption but $c$ is assumed to be coercive. Indeed their assumptions are weaker than the ones used in this work for $b$ and $m$ but stronger for $c$. 

This system is motivated by splitting methods in which we turn a single high 
order partial differential equation into a coupled system of lower order equations. For example, consider the PDE
\begin{equation} \label{eq:introExamplePDE}
Au=f,
\end{equation}
where $A$ is a fourth order differential operator. Suppose we may write $A=B_1\circ B_2 +C$, where $B_1,B_2$ and $C$ are second order differential operators. By introducing a new variable, $w=B_2 u$, we may rewrite \eqref{eq:introExamplePDE} as a coupled system of equations
\begin{equation} \label{eq:introExampleSplitting}
\begin{split}
C u + B_1 w & = f, \\
B_2 u - w & = 0.
\end{split}
\end{equation}
The advantage of such a splitting method is that the resulting system of equations is second order, it can thus be solved numerically using simpler finite elements than are required to directly solve \eqref{eq:introExamplePDE}. Of course the meaning of the resulting split systems also depends on boundary conditions where needed.
Please note that all examples presented in the paper are actually set on closed surfaces.
To be an effective method  the system \eqref{eq:introExampleSplitting} must itself be well posed. This question is considered in \cite{CiaHuaZou03}, where sharp conditions are given detailing well posedness of the system. Amongst these conditions is a relationship between the norm of $B_1-B_2$ and other properties of the operators (see \cite[Section 3.1]{CiaHuaZou03}). When designing a splitting method it can be difficult to ensure that this condition holds. 

In this paper  we will take $B_1=B_2$. This case is studied in \cite{KelLiu96,Yan02}. These papers treat the case where $C$ induces a bilinear operator that is coercive or at least positive semi definite. We will not make this assumption here as it is not compatible with many of problems we wish to consider. To illustrate this point, consider the operator
\[
A = \Delta^2 u +\Delta u + u.
\]    
Such an $A$ induces a coercive bilinear form on $H^2(\Omega)\cap H_0^1(\Omega)$ where $\Omega$ is a bounded open set in $\mathbb R^2$ with smooth boundary or on $H^2(\Gamma)$ where  $\Gamma$ is a closed smooth  hypersurface in $\mathbb R^3$. These lead to  a problem of the form \eqref{eq:introExamplePDE} being well posed. However to perform a splitting which satisfies the conditions in \cite{KelLiu96,Yan02} we require a $B_1$ which induces a bilinear form satisfying an inf sup condition, equivalently $B_1$ is invertible in an appropriate sense, and a $C$ which induces a positive semi-definite bilinear form. A possible choice is $B_1 = B_2= -\Delta + \lambda$ for some $\lambda >0$ (with homogeneous Dirichlet boundary condition in the case that 
$\Omega \subset \mathbb R^2$ with a smooth boundary) but this produces
\[
C = A- B_1 \circ B_1 = (1+2\lambda) \Delta +(1-\lambda^2) 
\]     
which isn't positive semi definite for any $\lambda >0$. We will thus consider a situation where $C$ does not induce a positive semi-definite bilinear form. Note that this work is not a direct generalisation of the results in \cite{KelLiu96,Yan02}, whilst we consider a weaker condition on $C$ this is accommodated by a stronger condition on the operator which acts on $w$ in the second equation, chosen to be the negative identity map in \eqref{eq:introExampleSplitting}.

Our abstract setting and assumptions  are motivated by applications of this general theory to formulate a splitting method for surface PDE problems arising in models of biomembranes which
are posed over a sphere and a torus, \cite{EllFriHob17}. 
The complexity of the fourth order operator we wish to split, which results from the second variation of the Willmore functional, makes it difficult to formulate the splitting problem in such a way that  existing theory can be applied. Such a formulation may be possible but it is our belief that the method presented here is  straightforward to apply to this and similar problems. Moreover the additional assumptions we make on $b$ and $m$ are quite natural for the applications we consider.  See also \cite{EllFreMil89, EllGraHobKorWol16} for other possible applications to fourth order partial differential equations. 
\subsection*{Outline of paper}
In Section \ref{AbsSplit} we define an abstract saddle point  system consisting of two coupled variational equations in a Banach space setting using three bilinear forms  \{c,b,m\}. Well posedness is proved subject to Assumptions  \ref{def:genSplitSetupASSUME} and   \ref{def:liftedDiscreteSetup}.
An abstract finite element approximation is defined in Section 3. Natural error bounds are proved under approximation assumptions. Section \ref{SurfCalc} details some notation for surface calculus and surface finite elements. Section \ref{BilinearFormb} details results about a useful bilinear form $b(\cdot,\cdot)$ used in the examples of fourth order surface PDEs studied in later sections. Examples of two fourth order PDEs on closed surfaces satisfying the assumptions of Section \ref{AbsSplit} are given in Section \ref{PDEex} and the analysis of the  application of the surface finite element method to the saddle point problem is studied in Section \ref{PDEexDiscrete}. Finally a couple of numerical examples are given in Section \ref{NumExpts} which verify  the proved convergence rates.

\section{Abstract splitting problem}\label{AbsSplit}
We now introduce the coupled system on which the splitting method is based. Our abstract problem  is formulated in a Banach space setting. We will first define the spaces and functionals used and the required assumptions.
\begin{definition} \label{def:genSplitSetup} 
 Let $X,Y$ be reflexive Banach spaces and $L$ be a Hilbert space  with $Y \subset L$ continuously.  Let  \{c,b,m\} be  bilinear functionals  such that
\begin{align*}
&c:X\times X \rightarrow \mathbb{R} \text{, bounded and bilinear}, \\
&b:X\times Y \rightarrow \mathbb{R} \text{, bounded, bilinear},  \\
&m:L\times L \rightarrow \mathbb{R} \text{, bounded, bilinear, symmetric and coercive}.
\end{align*}
Let $f \in X^*$ and $g \in Y^*$.
\end{definition} 
Using this general setting we formulate the coupled problem. Note that we allow a non-zero right hand side in each equation, this is a generalisation of the motivating problem \eqref{eq:introExampleSplitting}.

\begin{problem} \label{prob:genSplittingProb}
With the spaces and functionals in Definition \ref{def:genSplitSetup}, find $(u,w) \in X \times Y$ such that
\begin{equation} \label{eq:genSplittingProb}
\begin{split}
c(u,\eta) + b(\eta,w) &= \langle f, \eta \rangle \quad \forall \eta \in X, \\
b(u,\xi) - m(w,\xi) &= \langle g, \xi \rangle \quad \forall \xi \in Y.
\end{split}
\end{equation}
\end{problem}
Throughout we assume the following inf sup and coercivity conditions on
the bilinear forms $b(\cdot,\cdot), c(\cdot,\cdot)$ and $m(\cdot,\cdot)$.
\begin{assumption} \label{def:genSplitSetupASSUME} 
~

\begin{itemize}
\item
There exist $\beta,\gamma >0$ such that 
\begin{equation} \label{eq:infSupsDefn}
\beta \|\eta\|_X \leq \sup_{\xi \in Y} \frac{b(\eta,\xi)}{\|\xi\|_Y} \;\; \forall \eta \in X \quad \text{and} \quad 
\gamma \|\xi\|_Y \leq \sup_{\eta \in X} \frac{b(\eta,\xi)}{\|\eta\|_X} \;\; \forall \xi \in Y.
\end{equation}
\item
There exists $C>0$ such that for all $(u,w) \in X \times Y$
\begin{equation} \label{eq:uwCoercivity}
b(u,\xi) = m(w,\xi) \; \forall \xi \in Y \implies C \|w\|_L^2 \leq c(u,u) + m(w,w).
\end{equation}

\end{itemize}

\end{assumption}

For existence we will make the additional assumption that the spaces $X$ and $Y$ can be approximated by sequences of finite dimensional spaces. 
Moreover we assume that such approximating spaces are sufficiently rich to satisfy an appropriate inf sup inequality. This assumption allows us to use a Galerkin approach.

 \begin{assumption}\label{def:liftedDiscreteSetup}
 We assume there exist sequences of finite dimensional approximating spaces $X_n \subset X$ and $Y_n \subset Y$. That is, for 
any $\eta \in X$ there exists a sequence $\eta_n \in X_n$ such that $\|\eta_n - \eta\|_X \rightarrow 0$, similarly for any $\xi \in Y$ there exists a sequence $\xi_n \in Y_n$ such that $\|\xi_n - \xi\|_Y \rightarrow 0$.

Moreover, we assume the discrete inf sup inequalities hold. That is there exist $\tilde{\beta},\tilde{\gamma}> 0$, independent of $n$, such that
\begin{align}
\tilde{\beta} \|\eta\|_X &\leq \sup_{\xi \in Y_n} \frac{b(\eta,\xi)}{\|\xi\|_Y} \;\; \forall \eta \in X_n, \quad \\ 
\tilde{\gamma} \|\xi\|_Y &\leq \sup_{\eta \in X_n} \frac{b(\eta,\xi)}{\|\eta\|_X} \;\; \forall \xi \in Y_n.
\end{align}
Finally, assume there exists a map $I_n:Y \rightarrow Y_n$ for each $n$, such that
\begin{equation} \label{eq:YTildeSupConv}
\begin{split}
&b(\eta_n, I_n \xi) = b(\eta_n, \xi) \quad \forall (\eta_n, \xi) \in X_n \times Y \\
&\sup_{\xi \in Y} \frac{\|\xi - I_n \xi\|_L}{\|\xi\|_Y} \rightarrow 0 \quad \text{ as } n \rightarrow \infty.
\end{split}
\end{equation}

 \end{assumption}    
We now show the well posedness of Problem \ref{prob:genSplittingProb}. First we prove two key lemmas in which we 
construct a discrete inverse operator and a discrete coercivity relation that is an analogue of  (\ref{eq:uwCoercivity}).  We  make use of a generalised form of the 
Lax-Milgram theorem, the Banach-Ne\v{c}as-Babu\v{s}ka Theorem \cite[Section 2.1.3]{ErnGue13}. For completeness, the theorem is stated below.

\begin{theorem}[Banach-Ne\v{c}as-Babu\v{s}ka] \label{thm:BNB}
Let $W$ be a Banach Space and let $V$ be a reflexive Banach space. Let $A \in \mathcal{L}(W \times V ; \mathbb{R})$ and $F \in V^*$. Then there exists a unique $u_F \in W$ such that 
\[
A(u_F,v) = F(v) \;\; \forall v \in V
\] 
if and only if 
\begin{align*}
&\exists \alpha > 0 \;\; \forall w \in W, \quad  \sup_{v \in V} \frac{ A(w,v) }{ \|v\|_V } \geq \alpha \|w\|_W, \\
&\forall v \in V, \quad \left( \forall w \in W, \; A(w,v)=0 \right) \implies v = 0.
\end{align*}
Moreover the following a priori estimate holds
\[
\forall F \in V^*, \quad \|u_F\|_W \leq \alpha^{-1}\|F\|_{V^*}.
\]
\end{theorem}


\begin{lemma} \label{lem:liftedDiscreteCoercivity}
Under the Assumptions  \ref{def:genSplitSetup} and   \ref{def:liftedDiscreteSetup}, 
there exists a linear map $G_n:Y^* \rightarrow X_n$ such that for each $\Theta \in Y^*$ 
\[
b(G_n \Theta, \xi_n) = \langle \Theta , \xi_n \rangle \;\; \forall \xi_n \in Y_n.
\]
These maps satisfy the uniform bound
\[
\|G_n \Theta\|_X \leq \tilde{\beta}^{-1} \|\Theta\|_{Y^*}.
\]
Furthermore, there exists a map $G: Y^* \rightarrow X$ such that for each $\Theta \in Y^*$
\[
b(G \Theta, \xi) = \langle \Theta , \xi \rangle \;\; \forall \xi \in Y .
\]
This map satisfies the bound
\[
\|G \Theta\|_X \leq \beta^{-1} \|\Theta\|_{Y^*}.
\]
\end{lemma}
\begin{proof}
To construct $G_n$, let $\Theta \in Y^*$, then $\Theta \in (Y_n, \|\cdot\|_Y)^*$. Then by Theorem \ref{thm:BNB}, there exists a unique $G_n \Theta \in X_n$ such that 
\[
b(G_n \Theta, \xi_n) = \langle \Theta , \xi_n \rangle \;\; \forall \xi_n \in Y_n.
\]
The assumptions required to apply Theorem \ref{thm:BNB} are made in Assumption \ref{def:liftedDiscreteSetup}. That $G_n$ is linear follows immediately from the construction. The two bounds are a consequence of the discrete inf sup inequalities in Assumption \ref{def:liftedDiscreteSetup}. The map $G$ is constructed similarly using the assumptions made in Assumption \ref{def:genSplitSetupASSUME}.
\end{proof}

We can now prove a discrete coercivity relation which is key in proving well posedness for Problem \ref{prob:genSplittingProb}. This is a discrete analogue of \eqref{eq:uwCoercivity}.  

\begin{lemma}
Under the assumptions in Lemma \ref{lem:liftedDiscreteCoercivity}, there exists $C,N>0$ such that, for all $n\geq N$,
\begin{equation} \label{eq:liftedDiscreteCoercivity}
C\|v_n\|_L^2 \leq c(G_n m(v_n,\cdot),G_n m(v_n,\cdot)) + m(v_n,v_n) \;\; \forall v_n \in Y_n.
\end{equation}
Here $m(v_n,\cdot) \in Y^*$ denotes the map $y \mapsto m(v_n,y)$.
\end{lemma}
\begin{proof}
Let $v_n \in Y_n$, $m(v_n, \cdot) \in Y^*$ holds as $Y$ is continuously embedded into $L$ and observe
\[
\|m(v_n,\cdot)\|_{Y^*} = \sup_{y \in Y}\frac{|m(v_n,y)|}{\|y\|_Y} \leq C \frac{\|v_n\|_L \|y\|_L }{\|y\|_Y} \leq C \|v_n\|_L.
\]
It follows
\begin{align*}
b((G-G_n)m(v_n,\cdot),\xi) &= b((G-G_n)m(v_n,\cdot),\xi - I_n \xi) \\
&= b(G m(v_n,\cdot), \xi - I_n \xi)
\\
&= m(v_n, \xi - I_n \xi)
\\
& \leq C\|v_n \|_L \|\xi - I_n \xi\|_L.
\end{align*}
Using the inf sup inequalities given in \eqref{eq:infSupsDefn} we deduce
\[
\|(G-G_n)m(v_n,\cdot)\|_X \leq C \|v_n\|_L \sup_{\xi \in Y} \frac{\|\xi - I_n \xi\|_L}{\|\xi\|_Y}.
\]
For any $v_n \in Y_n$ we can thus bound the difference
\[
|c(G_n m(v_n,\cdot),G_n m(v_n,\cdot)) - c(G m(v_n,\cdot),G m(v_n,\cdot))| \leq C\|v_n\|_L^2 \sup_{\xi \in Y} \frac{\|\xi - I_n \xi\|_L}{\|\xi\|_Y}.
\]
Now, choosing $n$ sufficiently large in the bound above, by \eqref{eq:uwCoercivity} and \eqref{eq:YTildeSupConv} it follows for any $v_n \in Y_n$
\begin{align*}
C\|v_n\|_L^2 &\leq c(G_n m(v_n,\cdot),G_n m(v_n,\cdot)) + m(v_n,v_n) + \frac{C}{2}\|v_n\|_L^2,  
\end{align*}
from which the result is immediate.
\end{proof}
\begin{theorem} \label{thm:genSplittingWellPosed}
Suppose the Assumptions  \ref{def:genSplitSetup} and  \ref{def:liftedDiscreteSetup} hold, then there exists a unique solution to Problem \ref{prob:genSplittingProb}. Moreover, there exists $C>0$, independent of the data, such that
\[
\|u\|_X + \|w\|_Y \leq C(\|f\|_{X^*} + \|g\|_{Y^*}).
\]
\end{theorem}
\begin{proof}
We begin with existence, using a Galerkin argument. Let $(u_n,w_n) \in X_n \times Y_n$ be the unique solution of
\begin{align*}
c(u_n,\eta_n) + b(\eta_n,w_n) &= \langle f, \eta_n \rangle \quad \forall \eta_n \in X_n, \\
b(u_n,\xi_n) - m(w_n,\xi_n) &= \langle g,\xi_n \rangle \quad \forall \xi_n \in Y_n.
\end{align*}
As the problem is linear and finite dimensional, existence and uniqueness of such a solution is equivalent to uniqueness for the 
homogeneous problem $f=g=0$. In this case, testing the first equation with $u_n$, the second with $w_n$ and subtracting we obtain
\[
c(u_n,u_n) + m(w_n,w_n) = 0.
\] 
For sufficiently large $n$ this implies $w_n=0$ by \eqref{eq:liftedDiscreteCoercivity}, as $u_n = G_n m(w_n,\cdot)$ in the homogeneous 
case, thus $u_n=0$ also due to the linearity of $G_n$.

Now we return to the inhomogeneous case and produce a priori bounds on $u_n,w_n$. To create a pair of initial bounds 
we use the discrete inf sup inequalities with each of the finite dimensional equations. Firstly,
\begin{align*}
\tilde{\gamma}\|w_n\|_Y \leq \sup_{\eta_n \in X_n} \frac{b(\eta_n,w_n)}{\|\eta_n\|_X} \leq \|f\|_{X^*} + C \|u_n\|_X.
\end{align*}
Similarly with the second equation,
\begin{align*}
\tilde{\beta}\|u_n\|_X \leq \sup_{\xi_n \in Y_n} \frac{b(u_n,\xi_n)}{\|\xi_n\|_Y} \leq \|g\|_{Y^*} + C\|w_n\|_L.
\end{align*}
Combining these two inequalities produces
\begin{equation} \label{eq:liftAPwithL2}
\|u_n\|_X + \|w_n\|_Y \leq C(\|f\|_{X^*} + \|g\|_{Y^*} + \|w_n\|_L).
\end{equation}
To bound the $\|w_n\|_L$ term we use the same approach of subtracting the equations as used to 
show uniqueness. In the inhomogeneous case this produces
\[
c(u_n,u_n) + m(w_n,w_n) = \langle f, u_n \rangle - \langle g, w_n \rangle. 
\]
Notice now $u_n=G_n m(w_n,\cdot) + G_n g$, and thus (\ref{eq:liftedDiscreteCoercivity}) yields
\begin{align*}
C\|w_n\|_L^2 &\leq c(u_n,u_n) + m(w_n,w_n) -c(u_n,G_n g) -c(G_n g,u_n) +c(G_n g,G_n g) \\
& \leq \|f\|_{X^*}\|u_n\|_X + \|g\|_{Y^*}\|w_n\|_Y + C(\|u_n\|_X+ \|G_ng\|_X)\|G_ng\|_X. 
\end{align*}
Recall, by Lemma \ref{lem:liftedDiscreteCoercivity},
\[
\|G_n g\|_X \leq \tilde{\beta}^{-1}\|g\|_{Y^*}.
\] 
Combining these two inequalities with \eqref{eq:liftAPwithL2} produces
\[
\|w_n\|_L^2 \leq C(\|f\|_{X^*} + \|g\|_{Y^*})(\|f\|_{X^*} + \|g\|_{Y^*} + \|w_n\|_L).
\]
Hence by Young's inequality we deduce
\[
\|w_n\|_L \leq C(\|f\|_{X^*} + \|g\|_{Y^*}),
\]
then inserting this bound into \eqref{eq:liftAPwithL2} produces
\[
\|u_n\|_X + \|w_n\|_Y \leq C(\|f\|_{X^*} + \|g\|_{Y^*} ).
\]
Thus $u_n$ and $w_n$ are bounded sequences in $X$ and $Y$ respectively, which are both reflexive Banach spaces, 
hence there exists a subsequence (which we continue to denote with a subscript $n$) such that 
\[
u_n \xrightharpoonup{\;\;X\;\;} u \quad \text{ and } \quad w_n \xrightharpoonup{\;\;Y\;\;} w,
\]  
for some weak limits $u \in X$ and $w \in Y$. We will show that this weak limit is a solution to Problem \ref{prob:genSplittingProb}. 
For any $\eta \in X$, there exists an approximating sequence $\eta_n \rightarrow \eta$ with each $\eta_n \in X_n$, it follows
\[
c(u,\eta) + b(\eta,w) = \lim_{n \rightarrow \infty} \left( c(u_n,\eta_n) + b(\eta_n,w_n) \right) = \lim_{n \rightarrow \infty } \langle f , \eta_n \rangle = \langle f, \eta \rangle. 
\]
We treat the second equation similarly, for any $\xi \in Y$ we may find a sequence $\xi_n \rightarrow \xi$ with each $\xi_n \in Y_n$ and
\[
b(u,\xi) - m(w,\xi) = \lim_{n \rightarrow \infty} \left( b(u_n,\xi_n) - m(w_n, \xi_n) \right) = \lim_{n \rightarrow \infty } \langle g , \xi_n \rangle = \langle g, \xi \rangle.
\]
Thus $(u,w)$ does indeed solve Problem \ref{prob:genSplittingProb}. Moreover, as $u,w$ are the weak limits of bounded sequences in reflexive 
Banach spaces they satisfy the same upper bound, that is
\[
\|u\|_X + \|w\|_Y \leq C(\|f\|_{X^*} + \|g\|_{Y^*} ).
\]
We complete the proof by proving uniqueness, as the system is linear it is sufficient to consider the homogeneous case $f=g=0$. In such a case $b(u,\xi) = m(w,\xi) \; \forall \xi \in Y$ and 
\[
c(u,u) + m(w,w) = 0.
\]
Then by \eqref{eq:uwCoercivity} we have $w=0$ and hence $u=0$.
\end{proof}

\section{Abstract finite element method}
In this section we formulate and analyse an  abstract finite element method to approximate the solution of Problem \ref{prob:genSplittingProb}. In our applications we wish to use a non-conforming finite element method  in the sense of using finite element spaces which are not subspaces of the function spaces. For example, we will approximate problems based on a surface $\Gamma$ via problems based on a discrete surface $\Gamma_h$. 

\begin{definition} \label{def:genFemSetUp}
Suppose, for $h>0$, $X_h$, $Y_h$ are finite dimensional normed vector spaces and there exist lift operators
\begin{align*}
l_h^X:X_h \rightarrow X \quad \text{ and }\quad l_h^Y:Y_h \rightarrow Y,
\end{align*}   
which are linear and injective, such that $X_h^l := l_h^X(X_h)$ and $Y_h^l := l_h^Y(Y_h)$ satisfy Assumption \ref{def:liftedDiscreteSetup}.
For $\eta_h \in X_h$ let $\eta_h^l := l_h^X(\eta_h) \in X_h^l$, similarly for $\xi_h \in Y_h$ let $\xi_h^l := l_h^Y(\xi_h) \in Y_h^l$. 

Let $c_h$, $b_h$, $m_h$ denote bilinear functionals such that
\begin{align*}
&c_h:X_h\times X_h \rightarrow \mathbb{R} \text{, bilinear}, \\
&b_h:X_h\times Y_h \rightarrow \mathbb{R} \text{, bilinear,}  \\
&m_h:Y_h\times Y_h \rightarrow \mathbb{R} \text{, bilinear and symmetric.}
\end{align*}
We will assume the following approximation properties, there exists $C>0$ and $k \in \mathbb{N}$ such that 
\begin{align*}
|c(\eta_h^l,\xi_h^l) - c_h(\eta_h,\xi_h)| &\leq Ch^k \|\eta_h^l\|_X \|\xi_h^l\|_X \quad \forall (\eta_h, \xi_h) \in X_h \times X_h, \\
|b(\eta_h^l,\xi_h^l) - b_h(\eta_h,\xi_h)| &\leq Ch^k \|\eta_h^l\|_X \|\xi_h^l\|_Y \quad \forall (\eta_h, \xi_h) \in X_h \times Y_h, \\
|m(\eta_h^l,\xi_h^l) - m_h(\eta_h,\xi_h)| &\leq Ch^k \|\eta_h^l\|_L \|\xi_h^l\|_L \quad \forall (\eta_h, \xi_h) \in Y_h \times Y_h.
\end{align*}
Finally, let $f_h \in X_h^*$ and $g_h \in Y_h^*$, where $X_h^*$ and $Y_h^*$ are the dual spaces of $X_h$ and $Y_h$ respectively, be such that
\begin{align*}
|\langle f,\eta_h^l \rangle- \langle f_h,\eta_h \rangle| & \leq Ch^k \|f\|_{X^*} \|\eta_h^l\|_X \quad \forall \eta_h \in X_h, \\
|\langle g, \xi_h^l\rangle - \langle g_h, \xi_h \rangle| & \leq Ch^k \|g\|_{Y^*}\|\xi_h^l\|_Y \quad \forall \xi_h \in Y_h.
\end{align*}
\end{definition}

The finite element approximation can now be formulated.
\begin{problem} \label{prob:genSplittingFem}
Under the assumptions of Definition \ref{def:genFemSetUp}, find $(u_h,w_h) \in X_h \times Y_h$ solving the discretised problem
\begin{align*}
c_h(u_h,\eta_h) + b_h(\eta_h,w_h) &= \langle f_h, \eta_h \rangle \quad \forall \eta_h \in X_h, \\
b_h(u_h,\xi_h) - m_h(w_h,\xi_h) &= \langle g_h, \xi_h \rangle \quad \forall \xi_h \in Y_h.
\end{align*}
\end{problem}

We now prove well posedness for the finite element method, Problem \ref{prob:genSplittingFem}, and produce a priori bounds for the solution.

\begin{theorem} \label{thm:genFemInfBound}
For sufficiently small $h$, there exists a unique solution to Problem \ref{prob:genSplittingFem}. Moreover, there exists a constant $C>0$, independent of $h$, such that
\[
\|u-u_h^l\|_X + \|w-w_h^l\|_Y \leq C \inf_{(\eta_h, \xi_h) \in X_h \times Y_h} \left( \|u-\eta_h^l\|_X + \|w-\xi_h^l\|_Y \right) + h^k(\|f\|_{X^*} + \|g\|_{Y^*} ).
\]  
\end{theorem}
\begin{proof}
For existence and uniqueness it is sufficient to prove existence for the homogeneous case  $f_h=g_h=0$ as the system is linear and finite dimensional. In the homogeneous case we see
\[
c_h(u_h,u_h) + m_h(w_h,w_h) = 0. 
\]
We will denote by $G_h^l:Y^* \rightarrow X_h^l$ the map constructed in Lemma \ref{lem:liftedDiscreteCoercivity} and also define $G_h:Y^* \rightarrow X_h$ by $G_h:= (l_h^X)^{-1}\circ G_h^l$. Notice also,
\begin{align*}
\tilde{\beta}\|u_h^l - G_h^l m(w_h^l,\cdot)\|_X &\leq \sup_{\xi_h \in Y_h} \frac{b(u_h^l-G_h^l m(w_h^l,\cdot), \xi_h^l)}{\|\xi_h^l\|_Y} \\
& \leq \sup_{\xi_h \in Y_h} \frac{ b(u_h^l,\xi_h^l) - b_h(u_h,\xi_h) + m_h(w_h,\xi_h) - m(w_h^l,\xi_h^l) }{\|\xi_h^l\|_Y} \\
& \leq Ch^k \|w_h^l\|_L. 
\end{align*}
The final line holds as $\|u_h^l\|_X \leq C \|w_h^l\|_L$ in the homogeneous case, using the second equation of the system. It follows, by \eqref{eq:liftedDiscreteCoercivity},
\begin{align*}
C\|w_h^l\|_L^2 &\leq c(G_h^l m(w_h^l,\cdot), G_h^l m(w_h^l,\cdot) ) + m(w_h^l,w_h^l) \\
& = c(u_h^l, u_h^l ) + m(w_h^l,w_h^l) - c_h(u_h,u_h) - m_h(w_h,w_h) \\
& \quad + c(G_h^l m(w_h^l,\cdot), G_h^l m(w_h^l,\cdot) ) - c(u_h^l,u_h^l) \\
& \leq \tilde{C}h^k \|w_h^l\|_L^2.
\end{align*}
Hence for $h$ sufficiently small $w_h^l = 0$ from which we deduce $u_h^l=0$ and hence $w_h=u_h=0$. Thus there exists a unique solution for sufficiently small $h$.
Now we prove the required error estimate. Let $\eta_h \in X_h$ and $\xi_h \in Y_h$ be arbitrary. Using the second equation and the discrete inf sup inequality it follows
\begin{align*}
\tilde{\beta} &\|u_h^l-\eta_h^l\|_X \leq \sup_{v_h \in Y_h} \frac{1}{\|v_h^l\|_Y} \left[ b(u_h^l-\eta_h^l, v_h^l) \right] \\
& = \sup_{v_h \in Y_h} \frac{1}{\|v_h^l\|_Y} \bigg[ b(u-\eta_h^l, v_h^l) - m(w-w_h^l,v_h^l) -\langle g, v_h^l \rangle + \langle g_h ,v_h \rangle \\
& \hspace{4.25cm} -b_h(u_h,v_h) + m_h(w_h,v_h) + b(u_h^l,v_h^l) -m(w_h^l,v_h^l) \bigg] \\
& \leq C \left[ \|u-\eta_h^l\|_X + \|w-\xi_h^l\|_Y + \|w_h^l-\xi_h^l\|_L +h^k(\|g\|_{Y^*} + \|u_h^l\|_X + \|w_h^l\|_L) \right].
\end{align*}
We can produce a similar bound using the first equation of the system
\begin{align*}
\tilde{\gamma} &\|w_h^l-\xi_h^l\|_Y \leq \sup_{v_h \in X_h} \frac{1}{\|v_h^l\|_X} \left[ b( v_h^l,w_h^l-\xi_h^l) \right] \\
& = \sup_{v_h \in X_h} \frac{1}{\|v_h^l\|_X} \bigg[ b(v_h^l,w-\xi_h^l) +c(u-u_h^l,v_h^l) -\langle f, v_h^l \rangle + \langle f_h, v_h \rangle \\
& \hspace{4.7cm} -b_h(v_h,w_h) - c_h(u_h,v_h) + b(v_h^l,w_h^l) +c(u_h^l,v_h^l) \bigg] \\
& \leq C \left[ \|u-\eta_h^l\|_X + \|w-\xi_h^l\|_Y + \|u_h^l-\eta_h^l\|_X +h^k(\|f\|_{X^*} + \|u_h^l\|_X + \|w_h^l\|_Y) \right].
\end{align*}
Combining these two estimates produces the bound
\begin{equation} \label{eq:genSumBoundWithL2}
\begin{split}
\|u_h^l-\eta_h^l\|_X + \|w_h^l-\xi_h^l\|_Y \leq C \Big[ & \|u-\eta_h^l\|_X + \|w-\xi_h^l\|_Y + \|w_h^l-\xi_h^l\|_L \\
& +h^k(\|f\|_{X^*} + \|g\|_{Y^*} + \|u_h^l\|_X + \|w_h^l\|_Y) \Big].
\end{split}
\end{equation}

To produce the result we must bound the $L$-norm term which appears here. To do so we will add the discrete equations together and use the discrete coercivity relation \eqref{eq:liftedDiscreteCoercivity}. Firstly consider
\begin{align*}
&|c_h(u_h-\eta_h, u_h-\eta_h) + b_h(u_h-\eta_h, w_h-\xi_h)| \\
&= |c(u-\eta_h^l, u_h^l-\eta_h^l) + b(u_h^l-\eta_h^l, w-\xi_h^l) -\langle f, u_h^l-\eta_h^l \rangle + \langle f_h , u_h-\eta_h \rangle \\
& \quad +c(\eta_h^l, u_h^l-\eta_h^l) + b(u_h^l-\eta_h^l, \xi_h^l) - c_h(\eta_h, u_h-\eta_h) - b_h(u_h-\eta_h, \xi_h)| \\
& \leq C \|u_h^l-\eta_h^l\|_X \left[ \|u-\eta_h^l\|_X + \|w-\xi_h^l\|_Y + h^k(\|f\|_{X^*} +\|\eta_h^l\|_X + \|\xi_h^l\|_Y)  \right].
\end{align*}
Treating the second equation similarly produces
\begin{align*}
&|b_h(u_h-\eta_h, w_h-\xi_h) - m_h(w_h-\xi_h, w_h-\xi_h)| \\
&= |b(u-\eta_h^l, w_h^l-\xi_h^l) - m(w-\xi_h^l, w_h^l-\xi_h^l) -\langle g , w_h^l-\xi_h^l \rangle + \langle g_h, w_h-\xi_h \rangle \\
& \quad +b(\eta_h^l, w_h^l-\xi_h^l) - m(\xi_h^l,w_h^l - \xi_h^l) - b_h(\eta_h, w_h-\xi_h) + m_h(\xi_h, w_h-\xi_h)| \\
& \leq C \|w_h^l-\xi_h^l\|_Y \left[ \|u-\eta_h^l\|_X + \|w-\xi_h^l\|_Y + h^k(\|g\|_{Y^*}+\|\eta_h^l\|_X + \|\xi_h^l\|_Y)  \right].
\end{align*}
Combining these two estimates with \eqref{eq:genSumBoundWithL2} produces
\begin{equation}\label{eq:initialFemSumBound}
|c_h(u_h - \eta_h,u_h-\eta_h) + m_h(w_h-\xi_h,w_h-\xi_h)| \leq C \left(\mathbb{B}^2 + \mathbb{B}\|w_h^l-\xi_h^l\|_L \right), 
\end{equation}
where the grouping of terms $\mathbb{B}$ is given by
\begin{equation} \label{eq:bigBracketB}
\begin{split}
\mathbb{B}:=&\|u-\eta_h^l\|_X + \|w-\xi_h^l\|_Y \\
& +h^k(\|f\|_{X^*} + \|g\|_{Y^*} + \|u_h^l\|_X + \|\eta_h^l\|_X + \|w_h^l\|_Y + \|\xi_h^l\|_Y ).
\end{split}
\end{equation}
The coercivity relation in \eqref{eq:liftedDiscreteCoercivity} gives
\[
C\|w_h^l-\xi_h^l\|_L^2 \leq c(G_h^lm(w_h^l-\xi_h^l,\cdot),G_h^l m(w_h^l-\xi_h^l,\cdot) ) + m(w_h^l-\xi_h^l,w_h^l-\xi_h^l), \\
\]
it follows
\begin{equation} \label{eq:L2starterBound}
\begin{split}
C\|w_h^l-\xi_h^l\|_L^2\leq& |c(u_h^l-\eta_h^l,u_h^l-\eta_h^l ) + m(w_h^l-\xi_h^l,w_h^l-\xi_h^l) \\
& - \left[ c_h(u_h -\eta_h,u_h-\eta_h ) + m_h(w_h-\xi_h,w_h-\xi_h) \right]| \\
&+ |c_h(u_h - \eta_h,u_h-\eta_h) + m_h(w_h-\xi_h,w_h-\xi_h)| \\
&+ |c(G_h^lm(w_h^l-\xi_h^l,\cdot),G_h^l m(w_h^l-\xi_h^l,\cdot) ) - c(u_h^l-\eta_h^l,u_h^l-\eta_h^l )|. 
\end{split} \raisetag{2.5\baselineskip}
\end{equation}
To proceed we bound the three terms appearing here. The first term is simply an approximation property,
\begin{gather}
\begin{split}
& |c(u_h^l-\eta_h^l,u_h^l-\eta_h^l ) + m(w_h^l-\xi_h^l,w_h^l-\xi_h^l) - \left[ c_h(u_h -\eta_h,u_h-\eta_h ) + m_h(w_h-\xi_h,w_h-\xi_h) \right]| \\
& \leq Ch^k \left( \|u_h^l-\eta_h^l\|_X^2 + \|w_h^l-\xi_h^l\|_Y^2 \right) \\
&\leq Ch^k \left( \mathbb{B}^2 + \mathbb{B}\|w_h^l-\xi_h^l\|_L + \|w_h^l-\xi_h^l\|_L^2 \right).
\end{split} \raisetag{1\baselineskip}
\label{eq:starter1stLine}
\end{gather}
The final line is true for sufficiently small $h$ and follows from \eqref{eq:genSumBoundWithL2}. The second term we have already bounded in \eqref{eq:initialFemSumBound}. For the final term notice
\begin{align*}
|c&(G_h^lm(w_h^l-\xi_h^l,\cdot),G_h^l m(w_h^l-\xi_h^l,\cdot) ) - c(u_h^l-\eta_h^l,u_h^l-\eta_h^l )| \\
&\leq C(\|G_h^lm(w_h^l-\xi_h^l,\cdot)\|_X + \|u_h^l-\eta_h^l\|_X)\|G_h^lm(w_h^l-\xi_h^l,\cdot)-(u_h^l-\eta_h^l)\|_X.
\end{align*}
To bound these terms first notice, by Lemma \ref{lem:liftedDiscreteCoercivity},
\[
\|G_h^l m(w_h^l-\xi_h^l,\cdot)\|_X \leq C\|m(w_h^l-\xi_h^l,\cdot)\|_{Y^*} \leq C\|w_h^l-\xi_h^l\|_L.
\]
We can then use the bound on $\|u_h^l - \eta_h^l\|_X$ established in \eqref{eq:genSumBoundWithL2} to produce
\begin{align*}
\|G_h^l m(w_h^l-\xi_h^l,\cdot)\|_X + \|u_h^l-\eta_h^l\|_X \leq C \big[ & \|u-\eta_h^l\|_X + \|w-\xi_h^l\|_Y + \|w_h^l-\xi_h^l\|_L \\
&+h^k(\|f\|_{X^*} + \|g\|_{Y^*} + \|u_h^l\|_X + \|w_h^l\|_Y) \big].
\end{align*}
For the second factor we first introduce $G_h^l(g_h^l)$, where $g_h^l \in (Y_h^l)^*$ is defined by
\[
\langle g_h^l, v_h^l \rangle := \langle g_h,v_h \rangle.
\]
Note that the map $G_h^l$ is well defined on $(Y_h^l)^*$, see the proof of Lemma \ref{lem:liftedDiscreteCoercivity}.
By the triangle inequality
\[
\|G_h^lm(w_h^l-\xi_h^l,\cdot)-(u_h^l-\eta_h^l)\|_X \leq \|G_h^l(m(w_h^l,\cdot)+g_h^l)-u_h^l\|_X + \|\eta_h^l -G_h^l(g_h^l + m(\xi_h^l,\cdot))\|_X.
\]
To bound each of these we use the discrete inf sup inequalities and the definition of $G_h^l$. Firstly,
\begin{align*}
&\tilde{\beta}\|G_h^l(m(w_h^l,\cdot)+g_h^l)-u_h^l\|_X \leq \sup_{v_h \in Y_h} \frac{b(G_h^l(m(w_h^l,\cdot)+g_h^l)-u_h^l,v_h^l)}{\|v_h^l\|_Y} \\
& = \sup_{v_h \in Y_h} \frac{1}{\|v_h^l\|_Y} \left[ -b(u_h^l,v_h^l) +b_h(u_h,v_h) - m_h(w_h,v_h) + m (w_h^l,v_h^l) \right] \\
& \leq Ch^k\left( \|u_h^l\|_X + \|w_h^l\|_Y \right).
\end{align*}
Similarly, for the second term
\begin{align*}
&\tilde{\beta}\|\eta_h^l - G_h^l(m(\xi_h^l,\cdot)+g_h^l)\|_X \leq \sup_{v_h \in Y_h} \frac{b(\eta_h^l - G_h^l(m(\xi_h^l,\cdot)+g_h^l),v_h^l)}{\|v_h^l\|_Y} \\
&=\sup_{v_h \in Y_h} \frac{1}{\|v_h^l\|_Y} \left[ \langle g, v_h^l \rangle - \langle g_h,v_h \rangle + m(w-\xi_h^l,v_h^l ) + b(\eta_h^l-u,v_h^l) \right] \\
& \leq C(h^k \|g\|_{Y^*} + \|u-\eta_h^l\|_X + \|w-\xi_h^l\|_Y).
\end{align*}
Thus combining these bounds we see 
\begin{align}
|c&(G_h^lm(w_h^l-\xi_h^l,\cdot),G_h^l m(w_h^l-\xi_h^l,\cdot) ) - c(u_h^l-\eta_h^l,u_h^l-\eta_h^l )| \leq C \left(\mathbb{B}^2 + \mathbb{B}\|w_h^l-\xi_h^l\|_L \right). \label{eq:starter3rdLine}
\end{align}
Now, inserting \eqref{eq:initialFemSumBound}, \eqref{eq:starter1stLine} and \eqref{eq:starter3rdLine} into \eqref{eq:L2starterBound} and considering sufficiently small $h$, to absorb the final term appearing in \eqref{eq:starter1stLine} into the left hand side, produces
\begin{align*}
&\|w_h^l-\xi_h^l\|_L^2\leq C \left( \mathbb{B}^2 + \mathbb{B} \|w_h^l-\xi_h^l\|_L \right).
\end{align*}
Thus by Young's inequality 
\[
\|w_h^l-\xi_h^l\|_L \leq C\mathbb{B}.
\]
Inserting this bound into \eqref{eq:genSumBoundWithL2} gives
\begin{align*}
\|u_h^l-\eta_h^l\|_X + \|w_h^l-\xi_h^l\|_Y \leq C \bigg[  \|u-\eta_h^l\|_X + \|w-\xi_h^l\|_Y + h^k(\|f\|_{X^*} + \|g\|_{Y^*})& \\
  + h^k( \|u_h^l\|_X + \|\eta_h^l\|_X + \|w_h^l\|_Y + \|\xi_h^l\|_Y )& \bigg].
\end{align*}
We can deduce an a priori estimate by setting $\eta_h=\xi_h=0$ as then
\[
\|u_h^l\|_X + \|w_h^l\|_Y \leq C \left[ \|u\|_X + \|w\|_Y  +h^k(\|f\|_{X^*} + \|g\|_{Y^*} + \|u_h^l\|_X + \|w_h^l\|_Y ) \right],
\]
hence using the estimate in Theorem \ref{thm:genSplittingWellPosed}, for sufficiently small $h$,
\begin{equation}
\|u_h^l\|_X + \|w_h^l\|_Y \leq C \left[ \|f\|_{X^*} + \|g\|_{Y^*} \right].
\label{a_priori_estimate}
\end{equation}
Using this bound and the triangle inequality gives
\begin{align*}
\|u-u_h^l\|_X &+ \|w-w_h^l\|_Y \leq \|u - \eta_h^l\|_X + \|w-\xi_h^l\|_Y + \|u_h^l - \eta_h^l\|_X + \|w_h^l-\xi_h^l\|_Y \\
& \leq C \left[ \|u-\eta_h^l\|_X + \|w-\xi_h^l\|_Y  +h^k(\|f\|_{X^*} + \|g\|_{Y^*} + \|\eta_h^l\|_X + \|\xi_h^l\|_Y) \right].
\end{align*}
A further application of the triangle inequality and the a priori estimate in Theorem \ref{thm:genSplittingWellPosed} produces 
\begin{align*}
\|\eta_h^l\|_X + \|\xi_h^l\|_Y &\leq \|u-\eta_h^l\|_X + \|w-\xi_h^l\|_Y + \|u\|_X + \|w\|_Y \\
& \leq \|u-\eta_h^l\|_X + \|w-\xi_h^l\|_Y + C(\|f\|_{X^*} + \|g\|_{Y^*}).
\end{align*}
Thus for sufficiently small $h$ we have
\[
\|u-u_h^l\|_X + \|w-w_h^l\|_Y  \leq C \left[ \|u-\eta_h^l\|_X + \|w-\xi_h^l\|_Y  +h^k(\|f\|_{X^*} + \|g\|_{Y^*}) \right].
\]

Now we obtain the required result by taking an infimum, as the left hand side is independent of $\xi_h$ and $\eta_h$.
\end{proof}
This bound forms the core of the error analysis in our applications. There we will have the existence of an interpolation operator which allows this infimum bound to be turned into an error bound of the form $Ch^\alpha$, for some $0 \leq \alpha \leq k$. Exactly how large this $\alpha$ can be depends upon the regularity of the solution $(u,w)$. We now introduce this error bound in this abstract setting.

\begin{corollary} \label{cor:abstractOrderOfConv}
Suppose there exist Banach spaces $\tilde{X} \subset X$, $\tilde{Y} \subset Y$ such that $(u,w) \in \tilde{X} \times \tilde{Y}$ and with each embedding being continuous. Further assume there exists $\tilde{C},\alpha >0$, independent of $h$, such that 
\[
\inf_{(\eta_h, \xi_h) \in X_h \times Y_h} \|u-\eta_h^l\|_X + \|w-\xi_h^l\|_Y \leq \tilde{C}h^\alpha\left(  \|u\|_{\tilde{X}} + \|w\|_{\tilde{Y}} \right).
\] 
Then, for sufficiently small $h$, there exists $C>0$, independent of $h$, such that 
\[
\|u-u_h^l\|_X + \|w-w_h^l\|_Y \leq C h^{\min\left\{\alpha,k\right\}} \left( \|u\|_{\tilde{X}} + \|w\|_{\tilde{Y}} + \|f\|_{X^*} + \|g\|_{Y^*} \right).
\]
\end{corollary}   

We can also establish higher order error bounds in weaker norms by using a duality argument similar to the Aubin-Nitsche trick. To do so we assume that $c(\cdot,\cdot)$ is symmetric and that the Banach spaces $X$ and $Y$ can be embedded into some larger Hilbert spaces which supply the appropriate weaker norms.

\begin{proposition} \label{prop:abstractDualityBound}
Under the assumptions of Corollary \ref{cor:abstractOrderOfConv}, further suppose $c(\cdot,\cdot)$ is symmetric and there exist Hilbert spaces $H$, $J$ such that $X \subset H$ and $Y \subset J$ with both embeddings being continuous. Let $(\psi,\varphi) \in X \times Y$ denote the unique solution to Problem \ref{prob:genSplittingProb} with right hand side
\[
\eta \mapsto \langle u-u_h^l, \eta \rangle_H \quad \text{  and  } \quad \xi \mapsto \langle w-w_h^l, \xi \rangle_J.
\]  
Assume that there exist Banach spaces $\hat{X} \subset X$ and $\hat{Y} \subset Y$ such that $(\psi,\varphi) \in \hat{X} \times \hat{Y}$ with both embeddings continuous and $\tilde{C}, \beta >0$ such that
\begin{equation} \label{eq:dualityApproxAssumption}
\inf_{(\eta_h, \xi_h) \in X_h \times Y_h} \|\psi-\eta_h^l\|_X + \|\varphi-\xi_h^l\|_Y \leq \tilde{C}h^\beta\left(  \|\psi\|_{\hat{X}} + \|\varphi\|_{\hat{Y}} \right).
\end{equation}
Finally assume the regularity result
\begin{equation} \label{eq:dualityRegAssumption}
\|\psi\|_{\hat{X}} + \|\varphi\|_{\hat{Y}} \leq \hat{C}(\|u-u_h^l\|_H + \|w-w_h^l\|_J).
\end{equation}
Then, for sufficiently small $h$, there exists $C>0$, independent of $h$, such that
\[
\|u-u_h^l\|_H + \|w-w_h^l\|_J \leq C h^{\min\left\{\alpha + \beta,k\right\}} \left( \|u\|_{\tilde{X}} + \|w\|_{\tilde{Y}} + \|f\|_{X^*} + \|g\|_{Y^*} \right).
\]
\end{proposition} 
\begin{proof}
Let $(\psi,\varphi)$ be as defined in the statement above. It follows, for any $(\eta_h, \xi_h) \in X_h \times Y_h$,
\begin{align*}
 &\langle  u-u_h^l, u-u_h^l \rangle_H + \langle w-w_h^l, w-w_h^l \rangle_J \\
 &= c(u-u_h^l,\psi -\eta_h^l) + b(u-u_h^l,\varphi - \xi_h^l) + b(\psi-\eta_h^l,w-w_h^l) -m(w-w_h^l,\varphi - \xi_h^l) \\
&\quad + \langle f,\eta_h^l \rangle - \langle f_h,\eta_h \rangle + \langle g,\xi_h^l \rangle - \langle g_h,\xi_h \rangle - c(\eta_h^l,u_h^l) +c_h(u_h, \eta_h) \\
&\quad -b(u_h^l,\xi_h^l) + b_h(u_h,\xi_h) -b(\eta_h^l,w_h^l) + b_h(\eta_h,w_h) + m(w_h^l,\xi_h^l) - m_h(\xi_h,w_h).  
\end{align*}
It follows, using the boundedness and approximation properties of the bilinear operators,
\begin{align*}
 \langle u&-u_h^l, u-u_h^l \rangle_H + \langle w - w_h^l, w - w_h^l \rangle_J \\
\leq& C \Big[ (\|\psi-\eta_h^l\|_X + \|\varphi-\xi_h^l\|_Y)(\|u-u_h^l\|_X + \|w-w_h^l\|_Y) \\
& \quad + h^k(\|f\|_{X^*} + \|g\|_{Y^*} )(\|\psi-\eta_h^l\|_X + \|\varphi - \xi_h^l\|_Y + \|\psi\|_X + \|\varphi\|_Y ) \Big].
\end{align*}
Taking the infimum with respect to $(\eta_h, \xi_h)$ gives
\begin{align*}
	&\| u - u_h^l \|_H^2 + \| w - w_h^l \|_J^2 \\
	&\leq C ( \| u - u_h^l \|_H + \| w - w_h^l \|_J)
	\left[ h^{\alpha + \beta} ( \| u \|_{\tilde{X}} + \| w \|_{\tilde{Y}}) + h^k ( \| f \|_{X^*} + \| g \|_{Y^*}) \right]
\end{align*}
The result is then deduced, for sufficiently small $h$, using Young's inequality.
\end{proof}

\section{Surface calculus and  surface finite elements}\label{SurfCalc}

In this section we establish some notation with respect to surface PDEs and surface finite elements and study a particular bilinear form associated with a positive definite second order elliptic operator.
\subsection{Surface calculus}
We follow the development in \cite{DziEll13}. Let $\Gamma$ be a closed (that is compact and without boundary) $C^k$-hypersurface in $\mathbb R^3$, where
$k$ is as large as needed but at most $4$. There is a bounded domain $U \subset \mathbb{R}^3$ such that $\Gamma$ is the boundary set of $U$.
The unit normal $\nu$ to $\Gamma$  
that points away from this domain is called the outward unit normal.
We define $P:=  \unit - \nu \otimes \nu$ on $\Gamma$ 
to be, at each point of $\Gamma$, the projection onto the corresponding tangent space.
Here $\unit$ denotes the identity matrix in $\mathbb{R}^{3}$. 
For a differentiable function $f$ on $\Gamma$ we define the tangential gradient by
\begin{equation*}
	\nabla_{\Gamma} f := P \nabla \overline{f},
\end{equation*}
where $\overline{f}$ is a differentiable extension of $f$ to an open neighbourhood of 
$\Gamma \subset \mathbb{R}^{3}$. Here, $\nabla$ denotes the usual gradient in $\mathbb{R}^{3}$.
The above definition only depends on the values of $f$ on $\Gamma$.
In particular, it does not dependent on the extension $\overline{f}$, 
see Lemma 2.4 in \cite{DziEll13} for more details. 
The components of the tangential gradient are denoted by $(\D 1 f, \D 2 f, \D {3} f)^T := \nabla_\Gamma f$.
For a differentiable vector field $v: \Gamma \rightarrow \mathbb{R}^3$ we define the divergence by
$\nabla_{\Gamma} \cdot v:= \D 1 v_1 + \D 2 v_2 + \D 3 v_3$.
For a twice differentiable function the Laplace-Beltrami operator is defined by
$$
	\Delta_\Gamma f := \nabla_\Gamma \cdot \nabla_\Gamma f.
$$

The extended Weingarten map $\mathcal{H} := \nabla_{\Gamma} \nu$ is symmetric and has zero eigenvalue
in the normal direction. The eigenvalues $\kappa_i$, $i=1, 2$, belonging to the tangential eigenvectors
are the principal curvatures of $\Gamma$. The mean curvature $H$ is the sum of the principal curvatures,
that is $H := \sum_{i=1}^2 \kappa_i = \mbox{trace}\;( \mathcal{H}) = \nabla_{\Gamma} \cdot \nu$. 
Note that our definition differs from the more common one by a factor of $2$.
We will denote the identity function on $\Gamma$ by $id_\Gamma$, that is $id_\Gamma(p) = p$ for all $p \in \Gamma$.
The mean curvature vector $H \nu$ satisfies $H \nu = - \Delta_\Gamma id_\Gamma$, see Section 2.3 in \cite{DecDziEll05}.

\subsection{Surface finite elements}
\label{sub_section_SFEM}
We will consider surface finite elements, \cite{DziEll13}.
We assume that the surface $\Gamma$ is approximated by a polyhedral hypersurface
$$
	\Gamma_h = \bigcup_{T \in \mathcal{T}_h} T,
$$ 
where $\mathcal{T}_h$ denotes the set of two-dimensional simplices in $\mathbb{R}^{3}$
which are supposed to form an admissible triangulation. 
For $T \in \mathcal{T}_h$ the diameter of $T$ is $h(T)$ and the radius of the largest ball contained in $T$ is $\rho(T)$.
We set $h := \max_{T \in \mathcal{T}_h} h(T)$ and assume that the ratio between $h$ and $\rho(T)$ is uniformly bounded (independently of $h$). 
We assume that $\Gamma_h$ is
contained in a strip $\mathcal{N}_\delta$ of width $\delta > 0$ around $\Gamma$ on which the
decomposition 
$$
	x = p + d(x) \nu(p), \quad p \in \Gamma 
$$ 
is unique for all $x \in \mathcal{N}_\delta$. Here, $d(x)$ denotes the oriented distance function to $\Gamma$,
see Section 2.2 in \cite{DecDziEll05}. This defines a map $x \mapsto p(x)$ from $\mathcal{N}_\delta$ onto 
$\Gamma$. We here assume that the restriction $p_{|\Gamma_h}$ 
of this map on the polyhedral hypersurface $\Gamma_h$
is a bijective map between $\Gamma_h$ and $\Gamma$. In addition, the vertices of the simplices 
$T \in \mathcal{T}_h$ are supposed to sit on $\Gamma$. The generation of these triangulations for torii is rather standard, see for example \cite{DziEll13}. 

The piecewise affine Lagrange finite element space on $\Gamma_h$ is
$$
	\mathcal{S}_h := \left\{ \chi \in C(\Gamma_h) \;|\; \chi_{T} \in P^1(T)\;
	\forall T \in \mathcal{T}_h \right\},
$$ 
where $P^1(T)$ denotes the set of polynomials of degree $1$ or less on $T$.
The Lagrange basis functions $\varphi_i$ of this space are uniquely determined by their values
at the so-called Lagrange nodes $q_j$, that is $\varphi_i(q_j) = \delta_{ij}$. 
The associated Lagrange interpolation of a continuous function $f$ on $\Gamma_h$ is defined by
$$
	I_h f := \sum_{i} f(q_i) \varphi_i.
$$

We now introduce the lifted discrete spaces. We will use the standard lift operator as constructed in \cite[Section 4.1]{DziEll13}.  The lift $f^l$ of a function $f: \Gamma_h \rightarrow \mathbb{R}$ onto $\Gamma$ is defined by 
$$f^l(x) := (f \circ p_{|\Gamma_h}^{-1})(x)$$
 for all $x \in \Gamma$.
The inverse map is denoted by $f^{-l} := f \circ p$. The lifted finite element space is 
$$\mathcal{S}_h^l := \left\{ \chi^l \;|\; \chi \in \mathcal{S}_h \right\}.$$ 
Finally, the lifted Lagrange interpolation $I_h^l: C(\Gamma) \rightarrow \mathcal{S}_h^l$ is given by
$I_h^l f := (I_h f^{-l})^l$.
In the next section we introduce a bilinear form $b$ on $\Gamma$ for which we prove that the lifted discrete spaces satisfy the conditions in Assumption \ref{def:liftedDiscreteSetup} when we set $X_h^l := Y_h^l := \mathcal{S}_h^l$.
To be more precise here, for a sequence of triangulated surfaces $(\Gamma_{h_n})_{n \in \mathbb{N}}$ with maximal diameter $h_n \searrow 0$ for $n \rightarrow \infty$
we set $X_n := X_{h_n}^l = \mathcal{S}_{h_n}^l$ and $Y_n := Y_{h_n}^l = \mathcal{S}_{h_n}^l$.

\section{A useful bilinear form $b(\cdot,\cdot)$}\label{BilinearFormb}

Throughout this section let
 $b:X \times Y \rightarrow \mathbb{R}$ be given by
\[
b(u,v) := \int_\Gamma \nablaG u \cdot \nablaG v + \lambda uv \;do
\] 
for appropriate Banach spaces $X$ and $Y$ and positive constant $\lambda$.
\subsection{Inf-sup conditions}

\begin{proposition} \label{prop:infSupsforb}
Suppose $1< p \leq 2 \leq q < \infty$ are chosen such that $1/p + 1/q =1$. Let $\lambda > 0$, 
$X=W^{1,q}(\Gamma)$ and  $Y=W^{1,p}(\Gamma)$.There exist $\beta, \gamma > 0$ such that 
\[
\beta \|\eta\|_X \leq \sup_{\xi \in Y} \frac{b(\eta,\xi)}{\|\xi\|_Y} \;\; \forall \eta \in X \quad \text{and} \quad 
\gamma \|\xi\|_Y \leq \sup_{\eta \in X} \frac{b(\eta,\xi)}{\|\eta\|_X} \;\; \forall \xi \in Y.
\]
\end{proposition} 
\begin{proof}
Consider the map $A:W^{1,p}(\Gamma) \rightarrow W^{1,q}(\Gamma)^*$ given, for each $u \in W^{1,p}(\Gamma)$ by
\[
A(u)[v]:= b(v,u). 
\]
Evidently $A$ is well-defined and linear, by H\"{o}lder's inequality it is also continuous. We will now show that it is an isomorphism, beginning with showing that $A$ is surjective. Consider the inverse Laplacian type map $T:L^2(\Gamma) \rightarrow H^2(\Gamma)$, where, for $f\in L^2(\Gamma)$,  $Tf \in H^1(\Gamma)$ is defined to be  the unique solution to 
\[
b(Tf,v) = \int_\Gamma fv \quad \forall v \in H^1(\Gamma).
\]    
That $T$ is well defined, continuous and a bijection follows by elliptic regularity. It is immediate that $T^{-1}=-\DeltaG + \lambda Id$. 
Now suppose $F \in W^{1,q}(\Gamma)^*$ and set $g:=T^*(F) \in L^2(\Omega)$, this is well defined as $W^{1,q}(\Gamma)^* \subset H^2(\Gamma)^*$. 
For any $\varphi \in C^\infty_0(\Gamma)$ and first order derivative $\D\alpha$ it holds
\begin{align*}
&\int_\Gamma g\D\alpha\varphi = \int_\Gamma g\D\alpha T^{-1} T \varphi \\
&=\int_\Gamma g \left( T^{-1} \D\alpha T \varphi - \nu_\alpha (2\mathcal{H}:\nablaG \nablaG T\varphi + \nablaG H \cdot \nablaG T\varphi) -\left[(2\mathcal{H}^2-H\mathcal{H})\nablaG T\varphi \right]_\alpha \right).
\end{align*}
The second line is due to a commutation relation for $\D \alpha$ and $\DeltaG$ which follows from \cite[Lemma 2.6]{DziEll13}. To be more explicit, by summing over 
repeated indices we obtain for a twice continously differentiable function $u$ on $\Gamma$
\begin{align*}
\D \alpha \DeltaG u &= \D \alpha \D \beta \D \beta u = \D \beta \D \alpha \D \beta u + (\mathcal{H}_{\beta \gamma} \nu_\alpha - \mathcal{H}_{\alpha \gamma} \nu_\beta) 
\D \gamma \D \beta u
\\
&= \D \beta ( \D \beta \D \alpha u + (\mathcal{H}_{\beta \gamma} \nu_\alpha - \mathcal{H}_{\alpha\gamma} \nu_\beta ) \D \gamma u)
+ (\mathcal{H}_{\beta \gamma} \nu_\alpha - \mathcal{H}_{\alpha \gamma} \nu_\beta) 
\D \gamma \D \beta u
\\
&= \DeltaG \D \alpha u + (\nu_\alpha \D \beta \mathcal{H}_{\beta \gamma} + \mathcal{H}_{\beta \gamma} \mathcal{H}_{\beta \alpha} - H \mathcal{H}_{\alpha \gamma}) \D \gamma u
+ 2 \nu_\alpha \mathcal{H}_{\beta \gamma} \D \beta \D \gamma u
- \nu_\beta \mathcal{H}_{\alpha \gamma} \D \gamma \D \beta u,
\end{align*} 
and
\begin{align*}
& \D \beta \mathcal{H}_{\beta \gamma} = \D \beta \D \gamma \nu_\beta = \D \gamma \D \beta \nu_\beta + (\mathcal{H}_{\gamma \rho} \nu_\beta - \mathcal{H}_{\beta \rho} \nu_\gamma) \D \rho \nu_\beta 
= \D \gamma H - \mathcal{H}_{\beta \rho} \mathcal{H}_{\rho \beta} \nu_\gamma,
\\
& - \nu_\beta \mathcal{H}_{\alpha \gamma} \D \gamma \D \beta u = - \nu_\beta \mathcal{H}_{\alpha \gamma} (\mathcal{H}_{\beta \rho} \nu_\gamma - \mathcal{H}_{\gamma \rho} \nu_\beta)
\D \rho u = \mathcal{H}_{\alpha \gamma} \mathcal{H}_{\gamma \rho} \D \rho u.
\end{align*}
It then follows that
\begin{align*}
&\int_\Gamma -g\D\alpha\varphi + H\nu_\alpha g \varphi
\\ &=\langle F, T\left(H\nu_\alpha \varphi + \nu_\alpha (2\mathcal{H}:\nablaG \nablaG T\varphi + \nablaG H \cdot \nablaG T\varphi) + \left[(2\mathcal{H}^2-H\mathcal{H})\nablaG T\varphi \right]_\alpha \right) \rangle \\
& \quad -\langle F, \D\alpha T\varphi \rangle.  
\end{align*}
Notice $T \in \mathcal{L}(L^q(\Gamma),W^{2,q}(\Gamma))$, $\D\alpha \in \mathcal{L}(W^{2,q}(\Gamma),W^{1,q}(\Gamma))$ and
thus we may extend the map $\varphi \mapsto -\langle F, \D\alpha T\varphi \rangle$ to $L^q(\Gamma)$ and that extension lies in $L^q(\Gamma)^*$. The first term may be treated in a similar manner. It follows there exists $g_\alpha \in L^p(\Gamma)$ such that 
\[
\int_\Gamma -g\D\alpha\varphi + H\nu_\alpha g \varphi = \int_\Gamma g_\alpha \varphi \quad \forall \varphi \in C^\infty_0(\Gamma).
\]    
Hence $g \in W^{1,p}(\Gamma)$.
Now, for the constructed $g \in W^{1,p}(\Gamma)$ it holds, for any $v \in H^2(\Gamma)$,
\[
\int_\Omega g(-\Delta v + \lambda v) = \int_\Omega T^*F T^{-1}v = \langle F, v \rangle.    
\]
Integrating the left hand side by parts and using density the above equation implies, for any $v \in W^{1,q}(\Gamma)$,
\[
A(g)[v] = \int_\Omega \nablaG g \cdot \nablaG v + \lambda gv = \langle F,v \rangle.
\]
Hence $A(g)=F$ and thus $A$ is surjective. To show $A$ is injective, suppose $A(u)=0$, then in particular,
\[
0=A(u)[Tu] = \int_\Gamma u^2 \implies u = 0.
\]
Thus $A$ is a bijection and by the bounded inverse theorem $A^{-1}$ is also bounded, it follows
\[
\|\xi\|_Y \leq \|A^{-1}\| \|A\xi\|_{X^*} \quad \forall \xi \in Y.
\]
Hence we obtain
\[
\|A^{-1}\|^{-1} \|\xi \|_Y \leq \sup_{\eta \in X } \frac{b(\eta,\xi)}{\|\eta\|_X }.
\] 
Additionally, $(A^*)^{-1}=(A^{-1})^*$ is bounded, thus similarly 
\[
\|(A^*)^{-1}\|^{-1} \|\eta\|_X \leq \sup_{\xi \in Y } \frac{A^*(\eta)[\xi]}{\|\xi\|_Y }.
\] 
Finally notice $A^*(\eta)[\xi]=A(\xi)[\eta]=b(\eta,\xi)$, completing the second inf sup inequality.
Here, we have implicitly made use of the canonical isomorphism between $X$ and $X^{**}$.
\end{proof}    

\subsection{Ritz projection}
For the approximation and uniform convergence conditions \eqref{eq:YTildeSupConv} related to our bilinear form $b(\cdot,\cdot)$ we will make use of the Ritz projection which is defined in the lemma below.
\begin{lemma} \label{lem:ritzProjection}
Suppose $\lambda > 0$, let $1<r\leq\infty$,  $X:=W^{1,r}(\Gamma)$ and  $Y:= W^{1,s}(\Gamma)$ 
where $1 \leq s < \infty$ is chosen such that $1/r + 1/s =1$.
For each $h>0$, let $X_h^l:=Y_h^l:=\mathcal{S}_h^l$. There exists a bounded linear map $\Pi_h:W^{1,r}(\Gamma) \rightarrow (\mathcal{S}_h^l,\|\cdot \|_{1,r})$ given by
\[
b(\Pi_h \varphi, v_h^l) = b(\varphi, v_h^l) \quad \forall v^l_h \in \mathcal{S}^l_h.
\] 
There exists $C(r)>0$, independent of $h$, such that 
\[
\| \Pi_h \psi\|_{1,r} \leq C(r) \|\psi\|_{1,r} \quad \forall \psi \in W^{1,r}(\Gamma).
\]
Finally, it holds that
\[
 \sup_{\psi \in W^{1,r}(\Gamma)} \frac{\| \psi - \Pi_h \psi \|_{0,2}}{\| \psi \|_{1,r}} \rightarrow 0 \quad \textnormal{as} \quad h \searrow 0.
\]
\end{lemma}
\begin{proof}
One can see the Ritz projection $\Pi_h$ is well defined as this is equivalent to the invertibility of $S + \lambda M$, where $S,M$ are the usual mass and stiffness matrices
for lifted finite elements. The linearity of $\Pi_h$ is obvious. It is straightforward to show
that $\| \Pi_h \psi \|_{1,2} \leq C(\lambda) \| \psi \|_{1,2}$ for all $\psi \in W^{1,2}(\Gamma)$. From formula (4.16) in \cite{Pow17}, we learn that
$\| \Pi_h \psi \|_{1,\infty} \leq C \| \psi \|_{1,\infty}$. From the interpolation of Sobolev spaces, see e.g.~Corollary 5.13 in \cite{BenSha88}, we can deduce
that $\| \Pi_h \psi \|_{1,r} \leq C \| \psi \|_{1,r}$ for all $2 \leq r \leq \infty$. 
Observe that $b( \eta, \Pi_h \psi) = b(\Pi_h \eta, \Pi_h \psi) = b(\Pi_h \eta, \psi)$. Then, using Proposition \ref{prop:infSupsforb} with $q=r$ and $p=s$ for $1< s \leq 2$, it follows that
$$
	\gamma \| \Pi_h \psi \|_{1,s} \leq \sup_{\eta \in W^{1,r}(\Gamma)} \frac{b( \eta, \Pi_h \psi)}{\| \eta \|_{1,r}} 
	\leq \sup_{\eta \in W^{1,r}(\Gamma)} \frac{b( \Pi_h \eta, \psi)}{\| \eta \|_{1,r}}
	\leq C \| \psi \|_{1,s}
$$
so that we indeed have $\| \Pi_h \psi \|_{1,r} \leq C \| \psi \|_{1,r}$ for all $1 < r \leq \infty$. 

We next show that for $2 \leq s < \infty$,
$$
	\inf_{v_h^l \in \mathcal{S}_h^l} \| \psi - v_h^l\|_{1,s} \leq C h^{2/s} \| \psi \|_{2,2} \qquad \forall \psi \in H^2(\Gamma).
$$
Using the equivalence of the norms on the surfaces $\Gamma$ and $\Gamma_h$, see \cite{Dem09}, we can lift the usual interpolation estimates for the Lagrange interpolation operator $I_h$
onto $\Gamma$. We hence obtain,
\begin{align*}
 \| \psi - I_h^l \psi \|_{1,s} &= \left( \sum_{T \in \mathcal{T}_h^l} \| \psi - I_h^l \psi \|_{1,s,T}^s \right)^{1/s}
 \\
 &\leq C \left( \sum_{T \in \mathcal{T}_h^l } |T|^{1 - s/2} h^s \| \psi \|_{2,2,T}^s \right)^{1/s},
\end{align*}
where we have summed over all curved triangles $T$ of the lifted triangulation $\mathcal{T}_h^l$ of $\Gamma_h$. 
Under the assumptions on the triangulation $\mathcal{T}_h$ made in Section \ref{sub_section_SFEM}, it holds that $Ch^2 \leq |T|$.
Hence, $|T|^{1 - s/2} \leq C h^{2 - s}$ and
\begin{align*}
 \| \psi - I_h^l \psi \|_{1,s} \leq C h^{2/s} \left( \sum_{T \in \mathcal{T}_h^l} \| \psi \|_{2,2,T}^s  \right)^{1/s}.
\end{align*}
Using the estimate $(a^s + b^s) \leq (a^2 + b^2)^{s/2}$, which holds for all $a,b \geq 0$, we finally conclude that
\begin{equation}
	\| \psi - I_h^l \psi \|_{1,s} \leq C h^{2/s} \left( \sum_{T \in \mathcal{T}_h^l} \| \psi \|_{2,2,T}^2 \right)^{1/2} = C h^{2/s} \| \psi \|_{2,2}.
	\label{interpolation_estimate}
\end{equation}

Now, for $\psi \in W^{1,r}(\Gamma) \subset L^2(\Gamma)$ with $1 < r \leq \infty$, let $\varphi \in H^2(\Gamma)$ be the solution to
$$
	b(\varphi, v) = \int_{\Gamma} (\psi - \Pi_h \psi) v \; do \qquad \forall v \in H^1(\Gamma).
$$
It follows that
$$
	\| \psi - \Pi_h \psi \|_{0,2}^2 = b(\varphi, \psi - \Pi_h \psi) = b(\varphi - v_h^l, \psi - \Pi_h \psi),
$$
where $v_h^l \in \mathcal{S}_h^l$ is arbitrary. For $1 < r \leq 2$ and $s := r /(r - 1) \in [2,\infty)$, we obtain
\begin{align}
	\| \psi - \Pi_h \psi \|_{0,2}^2 &\leq C \inf_{v_h^l \in \mathcal{S}_h^l} \| \varphi - v_h^l \|_{1,s} \| \psi - \Pi_h \psi \|_{1,r} \leq C h^{2/s} \| \varphi \|_{2,2} \| \psi \|_{1,r}
	\nonumber \\
	&\leq C h^{2/s} \| \psi - \Pi_h \psi \|_{0,2} \| \psi \|_{1,r}.
	\label{interpolation_estimate_2}
\end{align}
On the other hand, for $2 \leq r \leq \infty$, we can conclude that
\begin{align*}
	\| \psi - \Pi_h \psi \|_{0,2}^2 &\leq C \inf_{v_h^l \in \mathcal{S}_h^l} \| \varphi - v_h^l \|_{1,2} \| \psi - \Pi_h \psi \|_{1,2}
	\leq Ch \| \varphi \|_{2,2} \| \psi \|_{1,2} 
	\nonumber \\
	&\leq Ch \| \psi - \Pi_h \psi \|_{0,2} \| \psi \|_{1,2} \leq Ch \| \psi - \Pi_h \psi \|_{0,2} \| \psi \|_{1,r}.
\end{align*}
Hence, for any $1 < r \leq \infty$,
$$
	\sup_{\psi \in W^{1,r}(\Gamma)} \frac{\| \psi - \Pi_h \psi\|_{0,2}}{\| \psi \|_{1,r}} \rightarrow 0 \;\; \textnormal{as} \;\; h \searrow 0.
$$
\end{proof}

\noindent 

For the choices $X = W^{1,q}(\Gamma)$ and $Y=W^{1,p}(\Gamma)$ with $1 < p \leq 2 \leq q < \infty$ such that $1/p + 1/q =1$ as well as $L = L^2(\Gamma)$, 
the uniform convergence condition \eqref{eq:YTildeSupConv} now follows by 
choosing $I_n := \Pi_{h_n}$ and setting $r=p$ in the lemma above. 
Furthermore, the conditions $\| \eta_n - \eta \|_X \rightarrow 0$ and $\| \xi_n - \xi \|_Y \rightarrow 0$ in Assumption \ref{def:liftedDiscreteSetup} hold for the following reasons.
First, $\eta$ and $\xi$ can be approximated sufficiently well by smooth functions $\tilde{\eta}$ and $\tilde{\xi}$, respectively.
Then, $\tilde{\eta}$ and $\tilde{\xi}$ are approximated by $I_h^l \tilde{\eta}$ and $I_h^l \tilde{\xi}$.
For $\tilde{\eta}$ this follows from (\ref{interpolation_estimate}) by choosing $s=q \geq 2$. 
For $\tilde{\xi}$ the estimate 
$\| \tilde{\xi} - I_h^l \tilde{\xi} \|_Y = \| \tilde{\xi} - I_h^l \tilde{\xi} \|_{1,p} \leq C \| \tilde{\xi} - I_h^l \tilde{\xi} \|_{1,2}
\leq C h \| \tilde{\xi} \|_{2,2}$ implies convergence.

\subsection{Discrete inf-sup condition}
To prove the discrete inf sup conditions we require Fortin's criterion. We use the following form of the criterion, which follows from \cite[Lemma 4.19]{ErnGue13}.
\begin{lemma} \label{lem:fortinsCriterion}
Suppose $V$ and $W$ are Banach spaces and $\tilde{b} \in \mathcal{L}(V \times W ; \mathbb{R})$ such that there exists $\beta > 0$ such that 
\[
\beta \leq \inf_{\xi \in W \setminus \left\{ 0\right\} } \sup_{\eta \in V \setminus \left\{ 0\right\}} \frac{\tilde{b}(\eta,\xi)}{\|\eta\|_V \|\xi\|_W}.
\]
Let $V_h \subset V$ and $W_h \subset W$ with $W_h$ reflexive.
If there exists $\delta > 0$ such that, for all $\eta \in V$, there exists $\Pi_h(\eta) \in V_h$  such that
\[
\forall \xi_h \in W_h, \quad \tilde{b}(\eta,\xi_h)=\tilde{b}(\Pi_h(\eta),\xi_h) \text{ and } \|\Pi_h(\eta)\|_V \leq \delta \|\eta\|_V,
\]
then 
\[
\frac{\beta}{\delta} \leq \inf_{\xi_h \in W_h \setminus \left\{ 0\right\} } \sup_{\eta_h \in V_h \setminus \left\{ 0\right\}} \frac{\tilde{b}(\eta_h,\xi_h)}{\|\eta_h\|_V \|\xi_h\|_W}.
\]
\end{lemma}

We can now prove the discrete inf sup conditions for $b(\cdot,\cdot)$.

\begin{lemma} \label{lem:discreteInfSupsForb}
Under the assumptions of Lemma \ref{lem:ritzProjection} (for $1< r < \infty$), there exist $\tilde{\beta}, \tilde{\gamma} > 0$, independent of $h$, such that 
\[
\tilde{\beta} \|\eta_h^l\|_X \leq \sup_{\xi_h \in Y_h} \frac{b(\eta_h^l,\xi_h^l)}{\|\xi_h^l\|_Y} \;\; \forall \eta_h^l \in X_h^l \quad \text{and} \quad 
\tilde{\gamma} \|\xi_h^l\|_Y \leq \sup_{\eta _h \in X_h} \frac{b(\eta_h^l,\xi_h^l)}{\|\eta_h^l\|_X} \;\; \forall \xi_h^l \in Y_h^l.
\]  
\end{lemma} 
\begin{proof}

We apply Fortin's Criterion (Lemma \ref{lem:fortinsCriterion}). Setting $V=W^{1,p}(\Gamma)$, $W=W^{1,q}(\Gamma)$, $V_h=W_h=\mathcal{S}_h^l$ and using the Ritz projection $\Pi_h$ constructed above in Lemma \ref{lem:ritzProjection} proves the first inf sup inequality.
Similarly, setting $W=W^{1,p}(\Gamma)$ and $V=W^{1,q}(\Gamma)$ proves the reversed inf sup inequality.
\end{proof}

\section{Applications to second order splitting of fourth order surface PDEs}\label{PDEex}

\subsection{A standard fourth order problem}
In this section we apply the abstract theory to splitting a fairly general fourth order surface PDE. That is we consider solving a problem of the form 
\[
\DeltaG^2 u - \nablaG \cdot (P\mathcal{B}P\nablaG u) 
+\mathcal{C}u = \mathcal{F},
\]
posed over $\Gamma \subset \mathbb{R}^3$, a closed 2-dimensional hypersurface. This PDE results from minimising the functional
\[
\frac{1}{2} \int_\Gamma (\DeltaG u)^2 + (\mathcal{B}\nablaG u) \cdot \nablaG u + \mathcal{C}u^2 -  2 \mathcal{F} u \;do
\]
(for symmetric $\mathcal{B}$) over $H^2(\Gamma)$.
We make the following assumptions on $\mathcal{B}$ and $\mathcal{C}$ to ensure that the equation is well posed. 
\begin{assumption}
Let  $\mathcal{B}: \Gamma \rightarrow \mathbb{R}^{3\times 3}$, $\mathcal B$  be  measurable and symmetric such that  there exists $\lambda_M >0$ satisfying
\[
\|\mathcal{B}(x)\| \leq \lambda_M \; \forall x \in \Gamma.
\]
Let $\mathcal{C}: \Gamma \rightarrow \mathbb{R}$ be  measurable and there exist $\mathcal{C}_m,\mathcal{C}_M >0$ such that
\[
\mathcal{C}_m < \mathcal{C}(x) < \mathcal{C}_M \;\; \forall x \in \Gamma.
\]
There exists $\Lambda >0$ such that 
\[
\frac{\Lambda \lambda_M}{2} < \mathcal{C}_m \quad \text{ and } \quad \frac{\lambda_M}{2\Lambda} < 1. 
\] 
Finally we suppose $\mathcal{F} \in L^2(\Gamma)$.
\end{assumption}

\begin{remark} 
Note that in the above $P\nabla_\Gamma u$ can be replaced by $\nabla_\Gamma u$ since $P$ projects onto the tangent space and that
$\nabla_\Gamma\cdot(P\mathcal B P\nabla_\Gamma u)=\nabla_\Gamma\cdot(\mathcal B \nabla_\Gamma u) - H \mathcal{B}\nabla_\Gamma  u \cdot \nu$. Also we can write $\mathcal B$ rather than $P\mathcal B P$ provided for each $x\in \Gamma$, $\mathcal B:\mathcal T_x\rightarrow \mathcal T_x$. \end{remark}

The well-posedness of the PDE follows by consideration of the weak formulation of the problem.

\begin{problem} \label{prob:4oSplittingexample}
Find $u \in H^2(\Gamma)$ such that 
\[
\int_\Gamma \DeltaG u \DeltaG v + \mathcal{B} \nablaG u \cdot \nablaG v + \mathcal{C} uv \; do = \int_\Gamma \mathcal{F} v \;do \;\; \forall v \in H^2(\Gamma).
\]
\end{problem}

The assumptions we make on $\mathcal{B}$ and $\mathcal{C}$ ensure that the bilinear form is coercive on $H^2(\Gamma) \times H^2(\Gamma)$ and hence the problem is well posed by the Lax-Milgram theorem. Here we have chosen an $L^2$ right hand side, one could make a more general choice, however, we restrict to $L^2$ here as we will later show that in this case the numerical method attains the optimal order of convergence. 

We will now formulate an appropriate splitting method whose solution coincides with that of the fourth order problem. The coupled PDEs in distributional form are
\begin{align}
&-\Delta_\Gamma w+w -\nabla_\Gamma\cdot((P\mathcal BP-2\unit)\nabla_\Gamma u)+(\mathcal C-1)u=\mathcal F\\
& -\Delta_\Gamma u +u-w=0.
\end{align}
This motivates solving Problem \ref{prob:genSplittingProb} with the following definition of the data. Note that $\mathcal G=0$ for the above PDE system.

\begin{definition} \label{def:4oSplittingExample}
With respect to Definition \ref{def:genSplitSetup}, set $L=L^2(\Gamma)$ and $X=Y=H^1(\Gamma)$. Set the bilinear functionals  
\begin{align*}
c(u,v) :=& \int_{\Gamma} (\mathcal{B} -2\unit  )\nabla_{\Gamma} u \cdot \nabla_{\Gamma} v + \left(\mathcal{C}-1 \right) uv \;do, \\
b(u,v):=& \int_\Gamma \nablaG u \cdot \nablaG v + uv \;do ,  \\
m(w,v):=& \int_\Gamma wv \;do.
\end{align*}
Finally, take the data to be
\[
f := m(\mathcal{F},\cdot) \qquad \text{and} \qquad g:=m(\mathcal{G},\cdot),
\] 
with $\mathcal{F}, \mathcal{G} \in L^2(\Gamma)$.
\end{definition}   
We can now use the abstract theory to show well posedness for this problem.

\begin{proposition} \label{prop:4oWPandBound}
There exists a unique solution to Problem \ref{prob:genSplittingProb} with the spaces and functionals as chosen in Definition \ref{def:4oSplittingExample}. Moreover
for $\mathcal{B} \in W^{1, \infty}(\Gamma)$ we have the regularity result $u,w \in H^2(\Gamma)$ with the estimate
\[
\|u\|_{H^2(\Gamma)} + \|w\|_{H^2(\Gamma)} \leq C \left( \|\mathcal{F}\|_{L^2(\Gamma)} + \|\mathcal{G}\|_{L^2(\Gamma)} \right).
\]
Furthermore, when $\mathcal{G}=0$ the solution $u$ coincides with the solution of Problem \ref{prob:4oSplittingexample}.
\end{proposition}

\begin{proof}
For the well posedness we apply Theorem \ref{thm:genSplittingWellPosed}. The assumptions required in Definition \ref{def:genSplitSetup} are straightforward to check, the inf sup conditions conditions are established in Proposition \ref{prop:infSupsforb} ($\lambda=1,p=q=2$). For the coercivity relation \eqref{eq:uwCoercivity} notice that
\[
b(u,\xi) = m(w,\xi) \;\forall \xi \in Y \implies u \in H^2(\Gamma) \text{ and } w = -\DeltaG u + u,
\] 
hence we deduce 
\[
c(u,u) + m(w,w) = \int_\Gamma (\DeltaG u )^2 + \mathcal{B} \nablaG u \cdot \nablaG u + \mathcal{C} u^2 \; do \geq C \int_\Gamma ( \DeltaG u)^2 + u^2 \; do \geq C\|w\|_{0,2}^2.
\]
For the assumptions made in Assumption \ref{def:liftedDiscreteSetup}, we take the lifted discrete spaces described in the previous section and the required discrete inf sup inequalities follow from Lemma \ref{lem:discreteInfSupsForb}. Finally, \eqref{eq:YTildeSupConv} holds by Lemma \ref{lem:ritzProjection}.
 
We thus have well posedness by Theorem \ref{thm:genSplittingWellPosed}. The regularity estimate follows by applying elliptic regularity to each of the equations of the system. Finally, when $\mathcal{G}=0$, by elliptic regularity we have
\[
w = -\DeltaG u + u.
\] 
It follows, for any $v \in H^2(\Gamma)$,
\begin{align*}
\int_\Gamma \mathcal{F}v \;do = c(u,v) + b(v,w) = \int_\Gamma \DeltaG u \DeltaG v + \mathcal{B} \nablaG u \cdot \nablaG v + \mathcal{C} uv \; do.
\end{align*}

\end{proof}

\subsection{Clifford torus problems}

We now look to apply the above theory to produce a splitting method for a pair of fourth order problems, 
based around the second variation of the Willmore functional, posed on a Clifford torus $\Gamma = T(R,R\sqrt{2})$. 
The problems are derived and motivated in Section 6.1.2 of \cite{EllFriHob17}. In order to state the problems we need the following definitions.

\begin{definition} \label{def:torusSplitSetup}
With respect to Definition \ref{def:genSplitSetup}, set the spaces to be $L=L^2(\Gamma)$, $X=W^{1,q}(\Gamma)$ and $Y=W^{1,p}(\Gamma)$, where $1< p <2 < q < \infty$ such that $1/p + 1/q =1$. Let $\delta,\rho>0$ be sufficiently small.    
We set the bilinear functionals to be as follows,
\begin{align*}
&c(u,v):= r_1(u,v) + r_2(u,v),~ b(u,v):= \int_\Gamma \nablaG u \cdot \nablaG v + uv \;do,~ m(v,w):= \int_\Gamma vw \;do\\
&\mbox{where}\\
&r_1(u,v) := \frac{1}{\rho} \sum_{k=1}^K \int_\Gamma u g_k \;do \int_\Gamma v g_k \;do + \chi_{\text{con}}\frac{1}{\delta} \sum_{k=1}^N u(X_k)v(X_k), ~r_2(u, v) := \int_{\Gamma} \nabla_{\Gamma} u \cdot\mathcal B\nabla_{\Gamma} v  + \mathcal C  uv \; do,  \\
&\mbox{and}\\
&\mathcal B:=  \left[\frac{3}{2}H^2 - 2|\mathcal{H}|^2 -2 \right]\unit - 2 H \mathcal{H}, \\& \mathcal C:=  - \frac{3}{2} H^2 | \mathcal{H} |^2 + 2 ( \nabla_{\Gamma} \nabla_{\Gamma} H) : \mathcal{H} + |\nabla_{\Gamma} H|^2 + 2 H Tr(\mathcal{H}^3) +\DeltaG |\mathcal{H}|^2 +|\mathcal{H}|^4 -1.\\
\end{align*}
Here the parameter $\chi_{\text{con}}$ takes one of 
two values leading to two problems. These are $\chi_ {\text{con}} =0$ or $\chi_{\text{con}}=1$ corresponding to the 
two cases of  point forces or point constraints  respectively. The functions $g_k$ are smooth and form a basis for the kernel of the second variation of the Willmore functional. Their specific form is given in Section 6.1.2 of \cite{EllFriHob17} but is not required here. Finally set $g=0$ and $f$ such that
\[
\langle f, v \rangle = \sum_{k=1}^N \beta_k v(X_k) \quad \text{ or } \quad \langle f, v \rangle = \frac{1}{\delta}\sum_{k=1}^N \alpha_k v(X_k),
\] 
for the point forces,  $\chi_ {\text{con}} =0$, or point constraints, $\chi_ {\text{con}} =1$, problem respectively.
\end{definition}

\begin{remark} The variational problem for the Clifford torus is to minimise over $H^2(\Gamma)$ the functional
$$\frac{1}{2}a(v,v)+\frac{1}{2} c(v,v) 
-\langle f,v \rangle$$
where
$$a(v,v):=(-\Delta_\Gamma v+v,-\Delta_\Gamma v+v) .$$
The  terms involving $\rho$ and $\delta$ in $r_1(\cdot,\cdot)$ are penalty terms which, respectively,  enforce orthogonality to the $\{g_k\}_{k=1}^K$ and point displacement constraints at $\{X_k\}_{k=1}^N$.
\end{remark}

We will now check that all of the assumptions required in Definition \ref{def:genSplitSetup} and Assumption \ref{def:genSplitSetupASSUME} hold for the choices made above in 
Definition \ref{def:torusSplitSetup}. Most of these are straightforward, however the inf sup conditions require the Proposition \ref{prop:infSupsforb}.
Now we check the remaining assumptions required.
\begin{lemma}
The assumptions made in Definition \ref{def:genSplitSetup} and Assumption \ref{def:genSplitSetupASSUME} hold for the choices made for the spaces and functionals in Definition \ref{def:torusSplitSetup}. 
\end{lemma}
\begin{proof}
The space $L^2(\Gamma)$ is a Hilbert Space and $W^{1,r}(\Gamma)$ is a reflexive Banach space for any $1<r<\infty$. The embedding $W^{1,p}(\Gamma) \subset L^2(\Gamma)$ is continuous by the Sobolev embedding theorem.

Having proven the inf sup inequalities in Proposition \ref{prop:infSupsforb}, the remaining conditions on $c,r,b$ and $m$ are straightforward. To obtain the coercivity relation \eqref{eq:uwCoercivity}, in this case from elliptic regularity
\[
b(u,\xi)=m(w,\xi) \;\forall \xi \in Y \implies  w=-\DeltaG u + u \text{ and } u \in H^2(\Gamma).
\]   
It follows 
\[
c(u,u) + m(w,w) = \int_\Gamma (\DeltaG u)^2 +2 |\nablaG u|^2 + u^2 +c(u,u) \geq C \|u\|^2_{2,2} \geq C\|w\|_{0,2}^2.
\]
The $H^2$ coercivity result used here holds for sufficiently small $\delta,\rho$, see in Proposition 5.2 and Section 6.1.2 of \cite{EllFriHob17}. 

Finally, the choices for $f$ and $g$ lie in the required dual spaces. For $f$ this follows from the continuous embedding $W^{1,q}(\Gamma) \subset C^0(\Gamma)$. 
\end{proof}

The splitting method is thus well posed, this follows by applying the abstract theory.

\begin{corollary} \label{cor:torusWPandReg}
There exists a unique solution to Problem \ref{prob:genSplittingProb} with the spaces and functionals as chosen in Definition \ref{def:torusSplitSetup}. Moreover we have the additional regularity $u \in W^{3,p}(\Gamma)$ for all $1<p<2$ and the regularity estimate
\[
\|u\|_{3,p} \leq C(p)\|w\|_{1,p}.
\]
\end{corollary}
\begin{proof}
We have proven that the assumptions made in Assumptions \ref{def:genSplitSetup} and \ref{def:liftedDiscreteSetup} hold in this case, thus we may apply Theorem \ref{thm:genSplittingWellPosed} to show well posedness. The regularity result follows by elliptic regularity, applied to the second equation of the system.
\end{proof}


\section{Second order splitting SFEM for fourth order surface PDEs}\label{PDEexDiscrete}

\subsection{Standard fourth order problem}
We now consider the standard fourth order problem and use the abstract theory to produce a convergent finite element method. Using $P^1$ finite elements, we will achieve optimal error bounds for both $u$ and $w$ of order $h$ convergence in the $H^1$ norm and order $h^2$ in the $L^2$ norm.  

\begin{definition} \label{def:4oProbFemSetUp}
In the context of Definition \ref{def:genFemSetUp}, set $X_h=Y_h=\mathcal{S}_h$. Take $l_h^X$ and $l_h^Y$ to be the standard lift operator, see Section \ref{sub_section_SFEM}.
Set the bilinear functionals to be
\begin{align*}
c_h(u_h,v_h) :=& \int_{\Gamma_h} ((P\mathcal{B}P)^{-l} -2\unit  )\nabla_{\Gamma_h} u_h \cdot \nabla_{\Gamma_h} v_h + \left(\mathcal{C}^{-l}-1 \right) u_hv_h \;do_h, \\
b_h(u_h,v_h):=& \int_{\Gamma_h} \nabla_{\Gamma_h} u_h \cdot \nabla_{\Gamma_h} v_h + u_hv_h \;do_h ,  \\
m_h(w_h,v_h):=& \int_{\Gamma_h} w_h v_h \;do_h.
\end{align*}
Here, $do_h$ denotes the induced volume measure on $\Gamma_h$.
Finally, set 
\[
 f_h := m_h(\mathcal{F}^{-l},\cdot) \quad \text{and} \quad g_h := m_h(\mathcal{G}^{-l},\cdot).
\] 
\end{definition}
We can now prove convergence for this method.

\begin{corollary} \label{cor:4oProbConvBounds}
With the spaces and functionals chosen in Definition \ref{def:4oSplittingExample} and Definition \ref{def:4oProbFemSetUp}, 
there exists $h_0 > 0$ such that for all $0<h<h_0$ there exists a unique solution $(u_h,w_h) \in X_h \times Y_h$ to the problem 
\begin{align*}
c_h(u_h,\eta_h) + b_h(\eta_h,w_h) &= \langle f_h, \eta_h \rangle \quad \forall \eta_h \in X_h, \\
b_h(u_h,\xi_h) - m_h(w_h,\xi_h) &= \langle g_h, \xi_h \rangle \quad \forall \xi_h \in Y_h.
\end{align*}
Moreover there exists $C>0$, independent of $h$, such that 
\[
\|u-u_h^l\|_{i,2} + \|w-w_h^l\|_{i,2} \leq C h^{2-i}(\|\mathcal{F}\|_{0,2} + \|\mathcal{G}\|_{0,2}),
\] 
for each $i=0,1$ and for all $0<h<h_0$.
\end{corollary}

\begin{proof}
For the $i=1$ case we apply Corollary \ref{cor:abstractOrderOfConv}, the assumptions on the lift operators and 
bilinear functionals made in Definition \ref{def:genFemSetUp} hold by the same arguments as for the Clifford torus application, see the proof of Corollary \ref{cor:oneMinusEpsBounds}. 
For the approximation to the data follow the proof of Lemma 4.7 in \cite{DziEll13},
\[
|m(\mathcal{F},\eta_h^l) - m_h(\mathcal{F}^{-l},\eta_h)| \leq C h^2|m(\mathcal{F},\eta_h^l)| \leq C h^2\|m(\mathcal{F},\cdot)\|_{X^*}\|\eta_h^l\|_{X},
\]
an identical argument holds for $\mathcal{G}$. Set the spaces $\tilde{X}=\tilde{Y}=H^2(\Gamma)$ and $\alpha=1$, the 
approximation assumption in Corollary \ref{cor:abstractOrderOfConv} holds by the standard interpolation estimates (see e.g. \cite[Lemma 4.3]{DziEll13}). It follows
\[
\|u-u_h^l\|_{1,2} + \|w-w_h^l\|_{1,2} \leq Ch\left(\|u\|_{2,2} + \|w\|_{2,2} + \|m(\mathcal{F},\cdot)\|_{-1,2} + \|m(\mathcal{G},\cdot)\|_{-1,2} \right).
\]
Hence by the regularity estimate in Proposition \ref{prop:4oWPandBound} we have 
\[
\|u-u_h^l\|_{1,2} + \|w-w_h^l\|_{1,2} \leq Ch\left(\|\mathcal{F}\|_{0,2} + \|\mathcal{G}\|_{0,2} \right).
\]

For the $i=0$ result we use Proposition \ref{prop:abstractDualityBound}, setting $H=J=L^2(\Gamma)$ and $\hat{X}=\hat{Y}=H^2(\Gamma)$. 
The approximation condition \eqref{eq:dualityApproxAssumption} holds for $\beta=1$ by the standard interpolation estimates. 
The regularity result \eqref{eq:dualityRegAssumption} holds by elliptic regularity applied to the dual problem. It follows
\[
\|u-u_h^l\|_{0,2} + \|w-w_h^l\|_{0,2} \leq Ch^2\left(\|\mathcal{F}\|_{0,2} + \|\mathcal{G}\|_{0,2} \right).
\]
\end{proof}

\subsection{Clifford torus problems}
We now apply the abstract finite element method to produce a convergent finite element approximation for the Clifford torus problems.

\begin{definition} \label{def:torusFemSetUp}
In the context of Definition \ref{def:genFemSetUp}, set $X_h=Y_h=\mathcal{S}_h$. Take $l_h^X$ and $l_h^Y$ to be the standard lift operator, see Section \ref{sub_section_SFEM}.
Set the bilinear functionals to be
\begin{align*}
&c_h(u_h,v_h) := \frac{1}{\rho} \sum_{k=1}^K \int_{\Gamma_h} u_h g_k \circ p \;do_h \int_{\Gamma_h} v_h g_k \circ p \;do_h + \chi_{\text{con}}\frac{1}{\delta} \sum_{k=1}^N u_h(p^{-1}(X_k))v_h(p^{-1}(X_k))  \\
& +  \int_{\Gamma_h} \nabla_{\Gamma_h} u_h \cdot \left( \left[\frac{3}{2}H^2 - 2|\mathcal{H}|^2 -2 \right]\unit - 2 H \mathcal{H} \right) \circ p \nabla_{\Gamma_h} v_h 
	\\  
	 & + u_hv_h \left( - \frac{3}{2} H^2 | \mathcal{H} |^2 + 2 ( \nabla_{\Gamma} \nabla_{\Gamma} H) : \mathcal{H} + |\nabla_{\Gamma} H|^2 + 2 H Tr(\mathcal{H}^3) +\DeltaG |\mathcal{H}|^2 +|\mathcal{H}|^4 -1\right) \circ p \; do_h,  \\
&b_h(u_h,v_h):= \int_{\Gamma_h} \nabla_{\Gamma_h} u_h \cdot \nabla_{\Gamma_h} v_h + u_hv_h \;do_h ,  \\
&m_h(w_h,v_h):= \int_{\Gamma_h} w_hv_h \;do_h.
\end{align*}
Finally, set $g_h = 0$ and  $f_h$ such that
\[
\langle f_h, v_h \rangle = \sum_{k=1}^N \beta_k v_h(p^{-1}(X_k)) \quad \text{ or } \quad \langle f_h, v_h \rangle = \frac{1}{\delta}\sum_{k=1}^N \alpha_k v_h(p^{-1}(X_k)).
\] 
\end{definition}
We shall check the assumptions made in Definition \ref{def:genFemSetUp} hold in this context and produce the following convergence result.

\begin{corollary} \label{cor:oneMinusEpsBounds}
With the spaces and functionals chosen in Definition \ref{def:torusSplitSetup} and Definition \ref{def:torusFemSetUp}, there exists $h_0 > 0$ such that for all $0<h<h_0$ there exists a unique solution $(u_h,w_h) \in X_h \times Y_h$ to the problem 
\begin{align*}
c_h(u_h,\eta_h) + b_h(\eta_h,w_h) &= \langle f_h, \eta_h \rangle \quad \forall \eta_h \in X_h, \\
b_h(u_h,\xi_h) - m_h(w_h,\xi_h) &= 0 \quad \forall \xi_h \in Y_h.
\end{align*}
Moreover, for any $1 < p < 2 < q < \infty$ with $1/p + 1/q = 1$ there exists $C(q) > 0$, independent of $h$, such that 
\[
\|u-u_h^l\|_{1,2} + \|w-w_h^l\|_{0,2} \leq C(q) h^{2/q} \|f\|_{X^*},
\] 
for all $0<h<h_0$.
\end{corollary}
\begin{proof}
Firstly, for the well posedness of the finite element method we need only check the assumptions made in Definition \ref{def:genFemSetUp} hold for the choices made in Definition \ref{def:torusFemSetUp}. The space $\mathcal{S}_h$ is a normed vector space and the standard lift operator is linear and injective, see \cite{DziEll13} for details. Each of the functionals defined are bilinear by inspection and $m_h$ is indeed symmetric. 

The approximation properties for $b_h$, $m_h$ and the $L^2$ and $H^1$ type terms in $c_h$ can be proven as in Lemma 4.7 of \cite{DziEll13}, in this case $k=2$.
The main idea is to compare the volume measures on $\Gamma$ and $\Gamma_h$ as well as the corresponding surface gradients $\nabla_\Gamma$
and $\nabla_{\Gamma_h}$. For the term with $\nabla_{\Gamma_h} u_h \cdot \mathcal{H} \circ p \nabla_{\Gamma_h} v_h$ in $c_h$, please keep in mind that $\mathcal{H}= P \mathcal{H} P$. 
Notice also we have treated $c_h$ analogously to the treatment of the surface diffusion term with symmetric mobility tensor in Section 3.1 of \cite{DziEll07}. For the remaining terms in $c_h$, the $1/\rho$ term can be treated in the same manner as the $L^2$ inner product and for the $1/\delta$ term observe
\[
\sum_{k=1}^N u_h(p^{-1}(X_k))v_h(p^{-1}(X_k)) = \sum_{k=1}^N u_h^l(X_k)v_h^l(X_k), 
\]
hence this term makes no contribution to the approximation error. A similar observation shows, in this case,
\begin{equation}
\langle f_h, v_h \rangle = \langle f, v_h^l \rangle.
\label{f_h_equals_f}
\end{equation}
Hence $f_h$ satisfies the required approximation property as does $g_h$ because $g_h=g=0$. We thus have satisfied all of the assumptions of Definition \ref{def:genFemSetUp}, hence the discrete problem is well posed by Theorem \ref{thm:genFemInfBound}. 

For the convergence result we will argue as in Proposition \ref{prop:abstractDualityBound}, however, 
due to the lack of further regularity in this circumstance a more careful argument is required.
Let $(\psi, \varphi) \in X \times Y$ denote the solution to Problem \ref{prob:genSplittingProb} with right hand side
$$
	\eta \mapsto \langle u - u_h^l, \eta \rangle_{H^1(\Gamma)} \quad \textnormal{and} \quad \xi \mapsto \langle w - w_h^l, \xi \rangle_{L^2(\Gamma)}.
$$
It follows
$$
\| u - u_h^l \|_{1,2}^2 + \| w - w_h^l \|_{0,2}^2
= c(\psi, u - u_h^l) + b(u - u_h^l, \varphi) + b(\psi, w - w_h^l) - m(\varphi, w - w_h^l).
$$
As $g = g_h = 0$ in this case we also have
$$
	b(u, \varphi) - m(w, \varphi) = 0,
$$
as well as
$$
	b_h(u_h, (\Pi_h \varphi)^{-l}) - m_h(w_h, (\Pi_h \varphi)^{-l}) = 0.
$$
Hence,
\begin{align*}
& | b( u - u_h^l, \varphi ) - m( w - w_h^l, \varphi ) | = | - b(u_h^l, \varphi) + m(w_h^l, \varphi) |
\\
&\leq | b_h(u_h, (\Pi_h \varphi)^{-l} ) - b(u_h^l, \varphi) | + | m(w_h^l, \varphi) - m(w_h^l, \Pi_h \varphi ) | + | m(w_h^l, \Pi_h \varphi) - m_h(w_h, (\Pi_h \varphi)^{-l} ) |
\\
&\leq C h^2 ( \| u_h^l \|_X \| \Pi_h \varphi \|_Y + \| w_h^l \|_L \| \Pi_h \varphi \|_L ) + C \| w_h^l \|_L \| \varphi - \Pi_h \varphi \|_L,
\end{align*}
where we have used the identity $b(\Pi_h \varphi, u_h^l) = b(\varphi, u_h^l)$ and the geometric estimates already discussed above, which produce the $h^2$ terms.
Then, using (\ref{interpolation_estimate_2}) for $r=p$ and (\ref{a_priori_estimate}), we obtain
\begin{align}
& | b( u - u_h^l, \varphi ) - m( w - w_h^l, \varphi ) |
\leq C h^2 ( \| u_h^l \|_X + \| w_h^l \|_Y ) \| \Pi_h \varphi \|_Y  + C \| w_h^l \|_Y \| \varphi - \Pi_h \varphi \|_L
\nonumber\\
&\leq C h^2 \| f \|_{X^*} \| \varphi \|_Y +  Ch^{2/q} \| f \|_{X^*} \| \varphi \|_{Y}
\label{estimate_ritz_projection_r_q}
\\
&\leq C h^{2/q} \| f \|_{X^*} ( \| u - u_h^l \|_{1,2} + \| w - w_h^l \|_{0,2} ). 
\nonumber
\end{align}
To deal with the two remaining terms, observe that for any $\eta_h \in \mathcal{S}_h$,
\begin{align*}
&| c(\eta_h^l, u - u_h^l) + b(\eta_h^l, w - w_h^l) | = | c(u, \eta_h^l) + b(\eta_h^l, w) - (c(u_h^l, \eta_h^l) + b(\eta_h^l, w_h^l))|
\\
&\leq | \langle f, \eta_h^l \rangle  - \langle f_h, \eta_h \rangle | + | c_h(u_h, \eta_h) + b_h(\eta_h, w_h) - (c( u_h^l, \eta_h^l) + b(\eta_h^l, w_h^l)) |
\\
&\leq Ch^2 \| u_h^l \|_X \| \eta_h^l \|_X  + Ch^2 \| \eta_h^l \|_X \| w_h^l \|_Y \leq Ch^2 \| f \|_{X^*} \| \eta_h^l \|_X,
\end{align*}
where we used (\ref{f_h_equals_f}) and 
the last step follows from (\ref{a_priori_estimate}). Choosing $\eta_h^l = I_h^l \psi$, the Lagrange interpolant, and $s = q$ in (\ref{interpolation_estimate}), we obtain
\begin{align}
 &|c(\psi, u - u_h^l) + b(\psi, w - w_h^l)|
 \leq | c(\psi - I_h^l \psi, u - u_h^l) + b(\psi - I_h^l \psi, w - w_h^l) | + Ch^2 \| f \|_{X^*} \| I_h^l \psi \|_X
 \nonumber \\
 &\leq C \| \psi - I_h^l \psi \|_X ( \| u - u_h^l \|_X + \| w - w_h^l \|_Y ) + C h^2 \| f \|_{X^*} \| I_h^l \psi \|_X
 \nonumber \\
 &\leq C h^{2/q}  \| \psi \|_{2,2} ( \| u - u_h^l \|_X + \| w - w_h^l \|_Y ) + C h^2 \| f \|_{X^*} \| \psi \|_{2,2}
 \label{estimate_lagrange_interpolation_s_q} \\
 &\leq C h^{2/q} ( \| \varphi \|_{0,2} + \| w - w_h^l \|_{0,2}) ( \| u - u_h^l \|_X + \| w - w_h^l \|_Y + \| f \|_{X^*} )
 \nonumber \\
 &\leq C h^{2/q} ( \| u - u_h^l \|_{1,2} + \| w - w_h^l \|_{0,2}) ( \| u - u_h^l \|_X + \| w - w_h^l \|_Y + \| f \|_{X^*} ),
 \nonumber
\end{align}
where we used the regularity of $\psi$ coming from the second equation, that is $b(\psi, \xi) = m(\varphi, \xi) + \langle w - w_h^l, \xi \rangle_{L^2(\Gamma)}$.
The a priori estimates from (\ref{a_priori_estimate}) and Theorem \ref{thm:genSplittingWellPosed} finally give 
$$
	|c(\psi, u - u_h^l) + b(\psi, w - w_h^l)| \leq C h^{2/q} \| f \|_{X^*} ( \| u - u_h^l \|_{1,2} + \| w - w_h^l \|_{0,2} )
$$
The result then follows by combining the estimates derived above.
\end{proof}

\section{Numerical examples}\label{NumExpts}

We conclude  with numerical examples showing that these theoretical convergence rates are achieved in practice. All of the numerical examples given here have been implemented in the DUNE framework, making particular use of the DUNE-FEM module \cite{DuneFem10}.

\subsection{Higher regularity problem}

We consider the problem outlined in Definition \ref{def:4oSplittingExample}, setting $\Gamma=S(0,1)$, the unit sphere, taking 
\[ \mathcal{B}(x)= \left( \begin{array}{ccc}
x_1 & 0 & 0 \\
0 & x_2 & 0 \\
0 & 0 & x_3 \end{array} \right),
\;\mathcal{C}(x)=2+x_1 x_2, \;C_m = 3/2, \; C_M = 5/2, \;\lambda_M = 1, \;\Lambda = 1,
\]
and selecting 
\begin{align*}
\mathcal{F}(x) &:= -5 x_3 (x_1^3 + x_2^3 + x_3^3) + 2x_3(x_1+x_2+x_3) -4x_3 + 4x_3^2 -1 +(1+x_1x_2)x_3 + 7x_1x_2, \\
\mathcal{G}(x) &:= 3x_3 - x_1 x_2. 
\end{align*}

These choices for $\mathcal{F}$ and $\mathcal{G}$ give the solution $(u,w)=(\nu_3,\nu_1\nu_2)$. The example is chosen as it shows that this method can be used to split a fourth order problem where the second order terms make an indefinite contribution to the bilinear form. Explicitly, the fourth order equation solved by $u$ is
\[
\Delta_\Gamma^2 u - \nabla_\Gamma \cdot \left( P\mathcal{B} P\nabla_\Gamma u \right)  + \mathcal{C} u 
= \mathcal{F} + \mathcal{G} - \Delta_\Gamma \mathcal{G}.
\]

The resulting errors and experimental orders of convergence are shown in Tables \ref{table:4oExampleUErrors} and \ref{table:4oExampleWErrors}. 
In each case, for grid size $h$, $E_V(h)$ is the error in the $V$ norm of the finite element approximation. For example, in Table \ref{table:4oExampleUErrors} we have
\[
E_{L^2(\Gamma)}(h) := \|u-u_h^l\|_{0,2}.
\] 
The experimental order of convergence ($EOC$) with respect to the $V$-norm, for tests with grid sizes $h_1$ and $h_2$, is given by
\[
EOC = \frac{\log(E_V(h_1)/E_V(h_2))}{\log(h_1/h_2)}.
\] 
In each of our examples the $EOC$ is calculated between the current $h$ and the previous refinement, so that the denominator is approximately $\log(1/2)$ each time as the grid size approximately halves with each refinement.
Observe that the method achieves the orders of convergence proven in Corollary \ref{cor:4oProbConvBounds}, order $h$ and $h^2$ convergence in the $H^1$ and $L^2$ norms respectively.

\begin{table}[ht]
\centering
\renewcommand{\arraystretch}{1.3}
\begin{tabular}{|c|c c|c c|}
\hline
$h$ & $E_{L^2(\Gamma)}(h)$ &  $EOC$ & $E_{H^1(\Gamma)}(h)$ & $EOC$ \\ \hline  
1.41421  &  $5.51463\times 10^{-1}$  &       -  &  $1.00111$  &       -  \\ \hline
7.07106$\times 10^{-1}$  &  $1.87559\times 10^{-1}$  & 1.55592   &  $6.28156\times 10^{-1}$  &  0.6724  \\ \hline
3.53553$\times 10^{-1}$  &  $5.05247\times 10^{-2}$  & 1.89228  &  $3.22169\times 10^{-1}$  &  0.963307  \\ \hline
1.76776$\times 10^{-1}$  &  $1.29659\times 10^{-2}$  & 1.96227   &  $1.59478\times 10^{-1}$  & 1.01446  \\ \hline
8.83883$\times 10^{-2}$  &  $3.2712\times 10^{-3}$  &  1.98683  &  $7.92569\times 10^{-2}$  & 1.00875  \\ \hline
4.41941$\times 10^{-2}$  &  $8.20352\times 10^{-4}$  &  1.9955  &  $3.95406\times 10^{-2}$  &  1.0032  \\ \hline
2.20970$\times 10^{-2}$  &  $2.05302\times 10^{-4}$  & 1.9985   &  $1.97563\times 10^{-2}$  & 1.00102  \\ \hline
1.10485$\times 10^{-2}$  &  $5.13428\times 10^{-5}$  & 1.99951  &  $9.87605\times 10^{-3}$  &  1.00031 \\ \hline
5.52427$\times 10^{-3}$  &  $1.28369\times 10^{-5}$  & 1.99987 &  $4.93772\times 10^{-3}$  & 1.00009   \\ \hline
\end{tabular}
\caption{Errors and Experimental orders of convergence for $u - u_h^l$. } \label{table:4oExampleUErrors}
\end{table}

\begin{table}[ht]
\centering
\renewcommand{\arraystretch}{1.3}
\begin{tabular}{|c|c c|c c|}
\hline
$h$ & $E_{L^2(\Gamma)}(h)$ &  $EOC$ & $E_{H^1(\Gamma)}(h)$ & $EOC$ \\ \hline  
1.41421  &  $8.14491\times 10^{-1}$  &       -  &  2.03684  &       -  \\ \hline
7.07106$\times 10^{-1}$  &  $5.01333\times 10^{-1}$  &  0.700128  &  1.30646  & 0.640664   \\ \hline
3.53553$\times 10^{-1}$  &  $1.739\times 10^{-1}$  &  1.52752  &  $6.64415\times 10^{-1}$  & 0.975509  \\ \hline
1.76776$\times 10^{-1}$  &  $4.73979\times 10^{-2}$  &  1.87536  &  $3.2514\times 10^{-1}$  & 1.03103  \\ \hline
8.83883$\times 10^{-2}$  &  $1.21214\times 10^{-2}$  &  1.96727  &  $1.61316\times 10^{-1}$  & 1.01117  \\ \hline
4.41941$\times 10^{-2}$  &  $3.04823\times 10^{-3}$  &  1.99151  &  $8.04912\times 10^{-2}$  & 1.00299   \\ \hline
2.20970$\times 10^{-2}$  &  $7.63212\times 10^{-4}$  &  1.99781  &  $4.02248\times 10^{-2}$  & 1.00075  \\ \hline
1.10485$\times 10^{-2}$  &  $1.90877\times 10^{-4}$  &  1.99944  &  $2.01098\times 10^{-2}$  & 1.00018  \\ \hline
5.52427$\times 10^{-3}$  &  $4.77244\times 10^{-5}$  &  1.99985  &  $1.00546\times 10^{-2}$  & 1.00005   \\ \hline
\end{tabular}
\caption{Errors and Experimental orders of convergence for $w - w_h^l$. } \label{table:4oExampleWErrors}
\end{table}

\subsection{Lower regularity problem}
We will next study a problem similar to the point forces problem on a Clifford torus introduced in Definition \ref{def:torusSplitSetup}. For ease of construction of an exact solution we will not study this problem precisely but a similar one on a sphere whose solution exhibits the same regularity, $(u,w) \in W^{3,p}(\Gamma) \times W^{1,p}(\Gamma)$ for any $1<p<2$, as proven in Corollary \ref{cor:torusWPandReg}.
The coupled problem we study, in distributional form, is given by
\begin{align*}
-\DeltaG w + w + \DeltaG u + 2u  &= \delta_N - \frac{1}{4\pi} - \frac{3}{4\pi}x_3 \\
-\DeltaG u + u -w &= \frac{3}{8\pi} \left[ (1-x_3)\log(1-x_3) + \frac{1}{2} - \log(2)  \right]
\end{align*}
where we take $\Gamma$ to be the unit sphere $\Gamma = S(0,1)$ and $\delta_N$ is a delta function centred at the north pole $N=(0,0,1)$. This can be viewed as a second order splitting of the fourth order PDE
\begin{align*}
\DeltaG^2 u - \DeltaG u + 3 u =& \delta_N -\frac{1}{4\pi}-\frac{3}{4\pi}x_3 + g - \DeltaG g
\end{align*}
with $g := \frac{3}{8\pi} \left[ (1-x_3)\log(1-x_3) + \frac{1}{2} - \log(2)  \right]$.

\begin{remark}  The construction of this problem follows by consideration of the function
\[
w(x) = -\frac{1}{4\pi}\left[ \log(1-x_3) -\log(2) +1 +\frac{3x_3}{2}  \right].
\]
This function has a smooth part and a logarithmic part which is based upon the Green's function for the Laplace Beltrami operator on a sphere, see \cite{KimOka87}. That is, in a distributional sense, $w$ satisfies
\[
-\DeltaG w = \delta_N - \frac{1}{4\pi} - \frac{3}{4\pi}x_3.
\]
The logarithmic part of $w$ lies in $W^{1,p}(\Gamma)$ for any $1<p<2$ but is not in $H^1(\Gamma)$. We take $u$ to be 
\[
u(x) = \frac{1}{8\pi} \left[ (1-x_3)\log(1-x_3) + \frac{1}{2} - \log(2)  \right].
\]

\end{remark}

The weak formulation and discretisation of the system is completely analogous to the treatment of the point forces problem described in Definition \ref{def:torusSplitSetup} and Definition \ref{def:torusFemSetUp}. Explicitly, in terms of the general abstract formulation in Problem \ref{prob:genSplittingProb}, we choose
\begin{align*}
& c(u,v):= \int_\Gamma -\nablaG u \cdot \nablaG v + 2uv \;do, 
 ~b(u,v):= \int_\Gamma \nablaG u \cdot \nablaG v + uv \;do, 
~ m(w,v):= \int_\Gamma wv \;do, \\
& \langle f,v \rangle := v(0,0,1) -\frac{1}{4\pi}\int_\Gamma v \;do - \frac{3}{4\pi}\int_\Gamma x_3 v \;do, 
~ \langle g,v \rangle := \int_\Gamma \frac{3}{8\pi} \left[ (1-x_3)\log(1-x_3) + \frac{1}{2} - \log(2)  \right] v \;do.
\end{align*}
The finite element method formulation is also completely analogous to the treatment of the point forces problem. Explicitly, in terms of the general abstract formulation in Problem \ref{prob:genSplittingFem}, we choose  
\begin{align*}
& c_h(u_h,v_h):= \int_{\Gamma_h} -\nabla_{\Gamma_h} u_h \cdot \nabla_{\Gamma_h} v_h + 2u_hv_h \;do_h, ~ b_h(u_h,v_h):= \int_{\Gamma_h} \nabla_{\Gamma_h} u_h \cdot \nabla_{\Gamma_h} v_h + u_hv_h \;do_h, \\
&m_h(w_h,v_h):= \int_{\Gamma_h} w_hv_h \;do_h, 
~ \langle f_h,v_h \rangle := v_h(p^{-1}(0,0,1)) -\frac{1}{4\pi}\int_{\Gamma_h} v_h \;do - \frac{3}{4\pi}\int_{\Gamma_h} (x_3)^{-l} v_h \;do_h, \\
& \langle g_h,v_h \rangle := \int_{\Gamma_h} \frac{3}{8\pi} \left[ (1-x_3)\log(1-x_3) + \frac{1}{2} - \log(2)  \right]^{-l} v_h \;do_h.
\end{align*}

The finite element method converges at the rates proven in Corollary \ref{cor:oneMinusEpsBounds}, where only the case $g = g_h = 0$ was addressed.
The experimental order of convergence is given in Tables \ref{table:deltaProbUErrors} and \ref{table:deltaProbWErrors}. 
In fact, we observe linear convergence in this example. The proof of this convergence rate is left for future research.
\begin{table}[h]
\centering
\renewcommand{\arraystretch}{1.3}
\begin{tabular}{|c|c c|c c|}
\hline
$h$ & $E_{L^2(\Gamma)}(h)$ &  $EOC$ & $E_{H^1(\Gamma)}(h)$ & $EOC$ \\ \hline  
1.41421  &  $7.2206\times 10^{-2}$  &       -  &  $9.60127\times 10^{-2}$  &       -  \\ \hline
7.07106$\times 10^{-1}$  &  $2.68314\times 10^{-2}$  & 1.4282   &  $4.81314\times 10^{-2}$  &  0.996248  \\ \hline
3.53553$\times 10^{-1}$  &  $7.71427\times 10^{-3}$  &  1.79832  &  $2.4304\times 10^{-2}$  &  0.985781  \\ \hline
1.76776$\times 10^{-1}$  &  $2.04304\times 10^{-3}$  & 1.91681   &  $1.24533\times 10^{-2}$  & 0.964672  \\ \hline
8.83883$\times 10^{-2}$  &  $5.30802\times 10^{-4}$  &  1.94447  &  $6.31331\times 10^{-3}$  & 0.980055 \\ \hline
4.41941$\times 10^{-2}$  &  $1.37634\times 10^{-4}$  &  1.94734  &  $3.17379\times 10^{-3}$  &  0.992192  \\ \hline
2.20970$\times 10^{-2}$  &  $3.57961\times 10^{-5}$  & 1.94296   &  $1.58979\times 10^{-3}$  & 0.997373 \\ \hline
1.10485$\times 10^{-2}$  &  $9.3513\times 10^{-6}$  & 1.93656  &  $7.95344\times 10^{-4}$  &  0.999182 \\ \hline
5.52427$\times 10^{-3}$  &  $2.45312\times 10^{-6}$  & 1.93055 &  $3.97739\times 10^{-4}$  & 0.999757   \\ \hline
\end{tabular}
\caption{Errors and Experimental orders of convergence for $u - u_h^l$. } \label{table:deltaProbUErrors}
\end{table}

\begin{table}[h]
\centering
\renewcommand{\arraystretch}{1.3}
\begin{tabular}{|c|c c|}
\hline
$h$ & $E_{L^2(\Gamma)}(h)$ &  $EOC$ \\ \hline  
1.41421  &  $1.28739\times 10^{-1}$  &       -   \\ \hline
7.07106$\times 10^{-1}$  &  $4.91831\times 10^{-2}$  &  1.38821   \\ \hline
3.53553$\times 10^{-1}$  &  $2.37553\times 10^{-2}$  &  1.04991  \\ \hline
1.76776$\times 10^{-1}$  &  $1.25937\times 10^{-2}$  &  0.915547  \\ \hline
8.83883$\times 10^{-2}$  &  $6.5736\times 10^{-3}$  &   0.937948  \\ \hline
4.41941$\times 10^{-2}$  &  $3.35583\times 10^{-3}$  &  0.970015 \\ \hline
2.20970$\times 10^{-2}$  &  $1.69215\times 10^{-3}$  &  0.987811 \\ \hline
1.10485$\times 10^{-2}$  &  $8.48703\times 10^{-4}$  &  0.995527 \\ \hline
5.52427$\times 10^{-3}$  &  $4.24803\times 10^{-4}$  &  0.998466 \\ \hline
\end{tabular}
\caption{Errors and Experimental orders of convergence for $w - w_h^l$. } \label{table:deltaProbWErrors}
\end{table}

\bibliography{refs}{}
\bibliographystyle{amsplain}

\end{document}